\documentclass[10pt, reqno]{amsart}
\usepackage{amsmath,amssymb,latexsym,esint,cite,mathrsfs}
\usepackage{verbatim,wasysym}
\usepackage[left=2.6cm,right=2.6cm,top=2.0cm,bottom=2.2cm]{geometry}
\usepackage{tikz,enumitem,graphicx, subfig, microtype, color}
\usepackage{epic,eepic}
\usepackage{amsmath}
\allowdisplaybreaks
\usepackage[colorlinks=true,urlcolor=blue, citecolor=red,linkcolor=blue,
linktocpage,pdfpagelabels, bookmarksnumbered,bookmarksopen]{hyperref}
\usepackage[hyperpageref]{backref}
\usepackage[english]{babel}
\usepackage{color}

\numberwithin{equation}{section}

\newtheorem{thm}{Theorem}[section]
\newtheorem{lem}[thm]{Lemma}

\newtheorem{prop}[thm]{Proposition}
\newtheorem{Def}[thm]{Definition}
\newtheorem{Rem}[thm]{Remark}

\newcommand{\R}{\mathbb{R}}

\def\lam {\lambda}

\def\sp {\quad}

\begin{document}
	\title[ Singular solutions ]{Symmetry of positive solutions to biharmonic Lane-Emden equation with singular set}
	
	\author[X.\ Huang]{Xia Huang}
	\author[Y.\ Li]{ Yuan Li}
	\author[X.\ Zhou]{Xianmei Zhou}

	\address{Xia Huang \newline\indent School of Mathematical Sciences and Shanghai Key Laboratory of PMMP, East China Normal University, \newline\indent
		Shanghai, 200241, People's Republic of China}
	\email{ xhuang@cpde.ecnu.edu.cn}

	\address{Yuan Li \newline\indent School of Mathematical Sciences and Shanghai Key Laboratory of PMMP, East China Normal University, \newline\indent
	Shanghai, 200241, People's Republic of China}
\email{ yli@math.ecnu.edu.cn}

	\address{ Xianmei Zhou \newline\indent School of Mathematical Sciences and Shanghai Key Laboratory of PMMP, East China Normal University, \newline\indent
		Shanghai, 200241, People's Republic of China}
	\email{  xmzhoumath@163.com}
	
	\subjclass[2010]{35J15, 45E10, 45G05}
	\keywords{Biharmonic Lane-Emden Equation; Weak/Punctured/Distributional Solutions; Singular Set; Symmetry and Monotonity}

	\begin{abstract}
		In this paper, we are devoted to studying the positive weak, punctured or distributional solutions to the biharmonic Lane-Emden equation
		\begin{equation*}
			\Delta^{2} u=u^{p}
			\quad
			\quad \text{in} \ \mathbb{R}^{N}\setminus  Z,	
		\end{equation*}
	  where $N\geq5$, $1<p\leq\frac{N+4}{N-4}$, and the singular set $Z$ represents a closed and proper subset of $ \left\lbrace x_{1}=0\right\rbrace $. The symmetry and monotonicity properties of the singular solutions will be given by taking advantage of the moving plane method and the approach of moving spheres.
		
	\end{abstract}
	\maketitle
	%
	
	\section{Introduction}
	In this paper, we focus on investigating the symmetry and monotonicity properties of the positive solutions to the following biharmonic Lane-Emden equation with singular set $Z$
	\begin{equation}\label{main3}
		\Delta^{2} u=u^{p}
		\quad
		\quad \text{in} \ \mathbb{R}^{N}\setminus  Z,	
	\end{equation}
	where $N\geq5$, $1<p\leq\frac{N+4}{N-4}$, $ Z$ is a closed and proper subset of the hyperplane $ \left\lbrace x_{1}=0\right\rbrace $.
In recent years, equation \eqref{main3} has been  extensively studied  in the huge of literatures, since they can be seen from  many different geometric and physical problems.

\medskip

To the best of our knowledge, various types of solutions to \eqref{main3} have been proposed and studied, including punctured solutions, weak solutions, distributional solutions and more. For the reader's convenience, let us precise the notion of solutions involve here. First, the function $u$ is called a {\bf punctured solution} to \eqref{main3} if $u\in C^{4}(\R^{N}\setminus  Z)$ and satisfies \eqref{main3} pointwisely  in the
classical sense. Secondly, $u$ is said a {\bf weak solution} to \eqref{main3} if $u\in H^{2}_{loc}(\R^{N}\setminus  Z)$, and \eqref{main3} is verified in the weak sense, that is,
 	\begin{equation}\label{u0}
	\int_{\R^N}\Delta u \Delta\varphi dx=	\int_{\R^N}u^{p} \varphi dx, \quad \forall\varphi\in C^{2}_{c}(\mathbb{R}^{N}\setminus  Z).
	\end{equation}	
 At last, we call $u$ a {\bf distributional solution} to $\eqref{main3}$ if $u^{p}\in L^{1}_{loc}(\R^N) $ and verifies \eqref{main3} in the sense of distributions, that is,
	\begin{equation*}
	\int_{\R^N} u \Delta^{2}\varphi dx=	\int_{\R^N}u^{p} \varphi dx, \quad \forall\varphi\in C^{\infty}_{c}(\mathbb{R}^{N}).
\end{equation*}	

\medskip

Now, we briefly review some of the literature on positive punctured and weak solutions to the  corresponding second order elliptic equation
\begin{equation}\label{le}
	-\Delta u=u^{p}
	\quad
	\quad \text{in} \ \mathbb{R}^{N}\setminus  Z,	
\end{equation}
where  $N\geq3$, $1<p\leq\frac{N+2}{N-2}$ and the singular set $Z$ is as above.
When $Z$ consists of only one single point, without loss of generality, setting the origin ${\{0}\}$, \eqref{le} is reduced to
\begin{equation}\label{le0}
	-\Delta u=u^{p}
	\quad
	\quad \text{in} \ \mathbb{R}^{N}\setminus  \left\lbrace 0\right\rbrace.
\end{equation}
If $x=0$ is a removable singularity, it is acknowledged that, Gidas and Spruck \cite{GS} demonstrated that \eqref{le0} has no positive  classical solutions in $ C^{2}(\R^N )$ for $1<p<\frac{N+2}{N-2}$. For the critical case  $p=\frac{N+2}{N-2}$, it's related to Yamabe equation,
 Gidas, Ni and Nirenberg \cite{GNN}(under the assumption on $u=O(|x|^{2-N})$ as $|x|\to+\infty$), as well as Caffarelli, Gidas and Spruck \cite{CGS}, respectively proved that any positive classical solution of \eqref{le0} has the form
\begin{equation*}
	u(x)=c\left( \frac{a}{1+a^{2}|x-x_{0}|^2}\right)^{\frac{N-2}{2}},\quad \text{for some}~ c,a>0, ~~ x_{0}\in\R^N .
\end{equation*}
 Later, Chen and Li \cite{CL} simplified these classification results by using the moving plane method, while Li \cite{Li} utilized the moving sphere method. On the other hand, if $x=0$ is a non-removable singularity point, i.e. $\lim\limits_{|x|\to 0}u(x)=+\infty$, then the positive solution to \eqref{le0} is radial
symmetry with respect to the origin with $1<p\leq\frac{N+2}{N-2}$, see \cite{CGS,CL} for punctured solution $u\in C^{2}(\R^N\setminus\left\lbrace 0\right\rbrace )$, and \cite{TS} for weak solution $u\in H^{1}_{loc}(\R^N\setminus\left\lbrace 0\right\rbrace )$. In particular, for $p=\frac{N+2}{N-2}$, under the Emden-Fowler transform \begin{equation*}
	v(t)=|x|^{\frac{N-2}{2}}u(x), \quad t=\log|x|,
\end{equation*}
Caffarelli, Gidas and Spruck \cite{CGS} also classified all those radial solutions of \eqref{le0} with non-removable singularity ${\{0}\}$ by ODE analysis, which are usually called the Fowler solutions or Delaunay type solutions.
Moreover, the above Delaunay type solutions play an important role in study of the singular Yamabe problem, the prescribed scalar curvature problem on Riemannian manifolds, and in obtaining a priori estimates in nonlinear equations. For the general case, if $Z$ is a closed and proper subset of the hyperplane $\left\lbrace x_{1}=0\right\rbrace $ satisfying ${ {\mathop{Cap_{2}}\limits_{\R^N}}( Z)}=0$, then the positive weak solution $u\in H^{1}_{loc}(\R^N\setminus Z )$ to \eqref{le} is symmetric with respect to the hyperplane $\left\lbrace x_1=0\right\rbrace $, see \cite{SB,EFS, CLin}, where Newton capacity defined by $${{\mathop{Cap_{2}}\limits_{\R^N}}( Z)}:=\inf\left\lbrace \int_{\R^N}|\nabla\varphi|^{2}dx: \varphi\in C^{\infty}_{c}(\R^N), \varphi\geq1\ \text{in a neighborhood of} \  Z \right\rbrace.$$

\medskip

For fourth order equation \eqref{main3},  when $Z=\left\lbrace 0\right\rbrace $, equation \eqref{main3} turns to
\begin{equation}\label{le2}
	\Delta^{2} u=u^{p}
	\quad
	\quad \text{in} \ \mathbb{R}^{N}\setminus  \left\lbrace 0\right\rbrace.
\end{equation}
 If $x=0$ is a removable singularity point, Lin\cite{Lin} proved that the unique nonnegative classical solution to \eqref{le2} in $C^{4}(\R^N)$ is trivial, i.e. $u\equiv 0$ for $1<p<\frac{N+4}{N-4}$ and any positive classical solutions of \eqref{le2} with $p=\frac{N+4}{N-4}$ has the form
\begin{equation*}
	u(x)=c\left( \frac{a}{1+a^{2}|x-x_{0}|^2}\right)^{\frac{N-4}{2}},\quad \text{for some}~ c,a>0, ~~ x_{0}\in\R^N .
\end{equation*}
 On the other hand, when $Z=\left\lbrace 0\right\rbrace $ is an non-removable singularity and $1<p\leq\frac{N+4}{N-4}$, Lin \cite{Lin} also showed that the positive punctured solution to \eqref{le2} is radially symmetric with respect to the origin via the method of moving planes. Then Guo, Huang, Wang and Wei\cite{GHWW}, Frank and K\"onig\cite{FK} classified the singular radial solution \eqref{le2} with critical exponent $p=\frac{N+4}{N-4}$,
by using the Emden-Flower transform
\begin{equation*}
	v(t)=|x|^{\frac{N-4}{2}}u(x), \quad t=\log|x|.
\end{equation*}
When $Z=\R^k$ is a lower dimensional hyperplane with $0\leq k\leq\frac{N-2m}{2}$, Du and Yang \cite{DY} showed the symmetry of the positive punctured solutions of the following conformal invariant biharmonic elliptic problems
\begin{equation}\label{main}
	\Delta^{2} u=u^{\frac{N+4}{N-4}}
	\quad
	\quad \text{in} \ \mathbb{R}^{N}\setminus Z.
\end{equation}

\medskip

An interesting question is whether the symmetry results of the positive punctured or weak solutions are perfectly valid for the fourth order problems if the singular set $Z$ is more general. Similar to the definition of capacity in \cite{EG}, we introduce the following definition of biharmonic capacity.
\begin{Def}
For each compact subset $ Z$ of $\Omega$, we define
 \begin{equation}\label{defcap}
	{\rm\mathop{Cap_{\Delta}}_{\Omega}( Z)}:=\inf\left\lbrace \int_{\Omega}|\Delta\varphi|^{2}dx: \varphi\in C^{\infty}_{c}(\Omega), \varphi\geq1\ \text{in a neighborhood of} \  Z \right\rbrace.
\end{equation}	
\end{Def}
Naturally, an important question arises: do the positive weak or punctured solutions of \eqref{main3} exhibit symmetry property under the condition ${\rm \mathop{Cap_{\Delta}}\limits_{\R^N}(Z)}=0 $?

Our main results are stated as follows.
\begin{thm}\label{symm2}
Let $N\geq5$ and $u\in H^{2}_{loc}(\R^{N}\setminus Z)$ be a positive weak solution to \eqref{main}. Assume that $u$ has non-removable  singularity in the singular set $Z$. If $Z$ is a closed and proper subset of the hyperplane $\left\lbrace x_{1}=0\right\rbrace $ satisfying
\begin{equation}\label{capz0}
\mathop{Cap_{\Delta}}_{\R^N}(\mathcal{K}_{z_0}(Z))=0,\quad\text{for some}\ z_0\notin Z,
\end{equation}
where $\mathcal{K}_{z_{0}}(Z)=\left\lbrace z_{0}+\frac{x-z_0}{|x-z_0|^{2}},\ x\in Z\right\rbrace $,
then $u$ is symmetric with respect to the  hyperplane $ \left\lbrace x_{1}=0\right\rbrace $ and increasing in the $x_{1}$-direction in $\left\lbrace x_{1}<0 \right\rbrace$. Furthermore, if $Z$ is a compact subset verifying \begin{equation}\label{capacity0}
	\mathop{Cap_{\Delta}}_{\R^N}( Z)=0,
\end{equation}
then the same conclusions hold for $u$.
\end{thm}
\begin{Rem}
 Indeed, if $Z$ is a compact subset of the hyperplane $ \left\lbrace x_{1}=0\right\rbrace $, we will see that  \eqref{capacity0} implies  \eqref{capz0} by Lemma \ref{kcap}.
\end{Rem}

 For above Theorem, focusing on the critical case for positive weak solutions, we employ the moving plane technique (as detailed in references \cite{AD,GNN,CL}) to derive the symmetry and monotonicity properties of the solutions to equation \eqref{main3}.

   To state our second symmetry results, we explore the punctured solutions via the upper Minkowski dimension, which is inspired by \cite{DY}. Let us recall the definition of the Minkowski dimension, see \cite{KLV,MP}.
\begin{Def}
Suppose that $E\subset \R^N$ is a compact set, the $\lambda$-dimensional Minkowski $r$-content of $E$ is defined by
	\begin{equation*}
		\mathcal{M}_r^\lambda(E)=\inf _{l \geq 1}\left\{l r^\lambda \mid E \subset \bigcup_{k=1}^l B\left(x_k, r\right), x_k \in E\right\},
	\end{equation*}
and that the upper and lower Minkowski dimensions are defined, respectively, as
\begin{equation*}
\overline{dim}_{M}(E)=\inf\left\lbrace\lambda\geq0\ \left| \limsup_{r\rightarrow 0} \mathcal{M}_r^\lambda(E)=0\right. \right\rbrace ,
\end{equation*}
and
\begin{equation*}
	\underline{dim}_{M}(E)=\inf\left\lbrace\lambda\geq0\ \left| \liminf_{r\rightarrow 0} \mathcal{M}_r^\lambda(E)=0\right. \right\rbrace.
\end{equation*}
If $\overline{dim}_{M}(E)=\underline{dim}_{M}(E)$, the common value is called the Minkowski dimensions ${dim}_{M}(E)$.
\end{Def}


\begin{thm}\label{symsub}
	Let $N\geq5$ and $u\in C^{4}(\R^{N}\setminus  Z)$ be a positive punctured solution to \eqref{main3} with $\frac{N}{N-4}<p\leq\frac{N+4}{N-4}$. Assume that $u$ has non-removable singularity in the singular set $ Z$.
If $ Z$ is a compact subset of the hyperplane $ \left\lbrace x_{1}=0\right\rbrace $ with $\overline{dim}_{M}( Z)<N-\frac{4p}{p-1}$  or is a smooth k-dimensional closed manifold with $k\leq N-\frac{4p}{p-1}$, then $u$ is symmetric with respect to the  hyperplane $ \left\lbrace x_{1}=0\right\rbrace $ and increasing in the $x_{1}$-direction in $\left\lbrace x_{1}<0 \right\rbrace$.
\end{thm}
\begin{Rem}
Let the singular set $Z$ be a compact subset of the hyperplane $ \left\lbrace x_{1}=0\right\rbrace $, then Proposition 4.2 in \cite{AGHW} implies that the condition $\overline{dim}_{M}( Z)<N-\frac{4p}{p-1}$ is stronger than ${\mathop{Cap_{\Delta}}\limits_{\R^N}}( Z)=0$.
\end{Rem}
As demonstrated in Proposition \ref{distribu}, a punctured solution to \eqref{main3} with compact subset $Z$  is always a distributional solution to \eqref{main3} in $\R^N$ when $p>\frac{N}{N-4}$, for the specific case $Z=\left\lbrace 0\right\rbrace $, it has been established in \cite{NY}. To be more precise, Ng\^o and Ye \cite{NY} have provided that a punctured solution $u$ to \eqref{le2} is a distributional solution to \eqref{le2} if and only if $p>\frac{N}{N-4}$. On the other hand,  Theorem 1.1 in \cite{NY} means that \eqref{main3} has no positive distributional solution for $1<p\leq\frac{N}{N-4}$, so $p>\frac{N}{N-4}$ is necessary for distributional solutions. Meanwhile, taking a proof step by step as the proof of Theorem \ref{symsub}, we can obtain the symmetry results for the positive distributional solution of \eqref{main3}.
\begin{thm}
	Let $N\geq5$ and $u\in L^{p}_{loc}(\R^N)\cap C^{1}(\R^{N}\setminus  Z)$ be a positive distributional solution to \eqref{main3} with $\frac{N}{N-4}<p\leq\frac{N+4}{N-4}$. Assume that $u$ has  non-removable singularity in the singular set $ Z$.
	If $ Z$ is a subset of the hyperplane $ \left\lbrace x_{1}=0\right\rbrace $ satisfying $\mathcal{L}^{N}(Z)=0$, where $\mathcal{L}^{N}(\cdot)$ is the Lebesgue measure, then $u$ is symmetric with respect to the  hyperplane $ \left\lbrace x_{1}=0\right\rbrace $ and increasing in the $x_{1}$-direction in $\left\lbrace x_{1}<0 \right\rbrace$.
\end{thm}
Regarding the subcritical (critical) case, our approach is inspired by \cite{JX, DY}.  Instead of studying the partial differential equation \eqref{main3}, we investigate the singular solutions of the global integral equation 	
\begin{equation}\label{integ}
	u(x)=c_{N,2}\int_{\R^N}\frac{u^{p}(y)}{|x-y|^{N-4}}dy, \quad \text{for } x\in\R^N\setminus Z,
\end{equation}
by Lemma \ref{integral}. Hence, we can prove Theorem \ref{symsub} via the moving spheres arguments of integral form \cite{Li}.
Recall that in the classification of positive solutions to equation \eqref{main3} with $Z=\left\lbrace 0\right\rbrace $, as studied by \cite{Lin}, it is important that every positive solutions of \eqref{main3} satisfies the superharmonic condition $-\Delta u\geq0$ in $\R^N\setminus\left\lbrace 0\right\rbrace $.  This sign condition is crucial for the application of the maximum principle, and remains invariant under the Kelvin transform. It is worthy noting that in our work, we do not assume the superharmonic condition $-\Delta u\geq0$ in $\R^N\setminus Z$.

\medskip

This paper is organized as follows. In Section 2, we study symmetry and monotonicity properties of weak solutions to the critical biharmonic equation \eqref{main}, by means of the moving plane technique. In Section 3, we investigate the
symmetry and monotonicity properties of punctured solutions for \eqref{main3} for subcritical (critical) case, with the help of the methods of moving spheres. In this paper, $c, C$ will be used to denote different constants.

\section{Symmetry results of weak solutions to equation \eqref{main}}
In this section, we aim to demonstrate the symmetry of the positive weak solutions for equation \eqref{main}. This will be derived from the subsequent theorem, which serves as our primary symmetry result.
\begin{thm}\label{symm}
	Let $N\geq5$ and $u\in H^{2}_{loc}(\R^{N}\setminus  Z)$ be a positive weak solution to \eqref{main}. Assume that $u$ has non-removable singularity in the singular set $ Z$, where $ Z$ is a compact subset of the hyperplane $ \left\lbrace x_{1}=0\right\rbrace $ such that \begin{equation}\label{capa0}
		\mathop{Cap_{\Delta}}_{\R^N}( Z)=0.
	\end{equation}
	Assume additionally that $u\in L^{2^{\ast\ast}}(\Sigma)$ with $2^{\ast\ast}=\frac{2N}{N-4}$ and $\Delta u\in L^{2}(\Sigma)$ for every $\Sigma$ which is a positive distance away from $ Z$.
	Then $u$ is symmetric with respect to the  hyperplane $ \left\lbrace x_{1}=0\right\rbrace $ and increasing in the $x_{1}$-direction in $\left\lbrace x_{1}<0 \right\rbrace$.
\end{thm}
In the following, we will prove Theorem \ref{symm} based on the moving plane method. First, let's fix some notations. For any $\lam<0$, define
\begin{equation*}
	\Sigma_{\lam}=\left\lbrace x=(x_{1},\cdots,x_{n})\ |\ x_{1}<\lam\right\rbrace,
\end{equation*}
and \begin{equation*}
	  x^{\lam}=(2\lam-x_{1},\cdots,x_{n})
\end{equation*}
which is the reflection through the hyperplane $T_{\lambda}:=\left\lbrace x_{1}=\lambda \right\rbrace $. Denote \begin{equation*}
 	 Z_{\lambda}=\left\lbrace x\in\Sigma_{\lam}\ |\ x^{\lambda}\in Z\right\rbrace, \sp	\Sigma_{\lam}^{\ast}=\Sigma_{\lam}\setminus Z_{\lambda},
 \end{equation*}
and \begin{equation*}
	u_{\lambda}(x)=u(x^{\lambda}),\sp w_{\lambda}(x)=u(x)-u_{\lambda}(x),\sp x\in\Sigma_{\lam}^{\ast}.
\end{equation*}
Let $u$ be a positive  weak solution of \eqref{main}, then we get
	\begin{equation}\label{u1}
	\int_{\R^N}\Delta u_{\lambda} \Delta\varphi dx=	\int_{\R^N}u_{\lambda}^{\frac{N+4}{N-4}} \varphi, \quad \forall\varphi\in C^{2}_{c}(\mathbb{R}^{N}\setminus  Z_{\lambda}).
\end{equation}	
Next we show some lemmas of vital importance in the proof of symmetry.
\begin{lem}\label{cap0}
	Let $\Omega\subset\R^N$ be an open bounded subset and $ Z\subset\Omega$ be a compact subset. Then there exists a constant $C_{1}>1$ such that
	\begin{equation*}
			\mathop{Cap_{\Delta}}_{\R^N}( Z)\leq\mathop{Cap_{\Delta}}_{\Omega}( Z)\leq C_{1}	\mathop{Cap_{\Delta}}_{\R^N}( Z).
	\end{equation*}
\end{lem}
\begin{proof}
	By the definition of \eqref{defcap}, we easily have that \begin{equation*}
		\mathop{Cap_{\Delta}}_{\R^N}( Z)\leq\mathop{Cap_{\Delta}}_{\Omega}( Z),
	\end{equation*}
and for any $\varepsilon>0$, there exists $\varphi_{\varepsilon}\in C^{\infty}_{c}(\R^N)$,  $\varphi_{\varepsilon}\geq1$ in a neighborhood of $ Z$, such that
 \begin{equation}\label{rnn}
\|\Delta\varphi_{\varepsilon}\|_{L^{2}(\R^N)}^{2}\leq\mathop{Cap_{\Delta}}_{\R^N}( Z)+\varepsilon .
\end{equation}
By the extension theorem (see section 5.4 in\cite{Evans}), we can choose an open subset $\Omega^{\prime}\subset\subset\Omega$ with $\partial\Omega^{\prime}$ being $C^{2}$  and a function $\widetilde{\varphi}_{\varepsilon}\in C^{\infty}_{c}(\Omega) $ such that
\begin{equation*}
	\widetilde{\varphi}_{\varepsilon}=\begin{cases} \varphi_{\varepsilon} & \text { in }  \Omega^{\prime} \\ 0 & \text { in } \R^N\setminus\Omega\end{cases},
\end{equation*}
and satisfying
\begin{equation}\label{extend}
	\|\widetilde{\varphi}_{\varepsilon}\|_{H^{2}(\R^N)}\leq C \|\widetilde{\varphi}_{\varepsilon}\|_{H^{2}(\Omega^{\prime})} ,
\end{equation}
where $C$ is a suitable positive constant depending on $\Omega^{\prime}$ and $\Omega$. It follows from the (\ref{defcap}) that
 \begin{equation*}
\begin{aligned}
\mathop{Cap_{\Delta}}_{\Omega}( Z)&\leq\|\widetilde{\varphi}_{\varepsilon}\|^{2}_{H^{2}(\R^N)}\leq C^{2}\|\widetilde{\varphi}_{\varepsilon}\|^{2}_{H^{2}(\Omega^{\prime})}\\
&\leq C^{2}\left( \|{\varphi}_{\varepsilon}\|^{2}_{L^{2}(\R^N)}+\|\Delta {\varphi}_{\varepsilon}\|^{2}_{L^{2}({\R^N})}\right) \\
&\leq C_{1}\left( \|{\varphi}_{\varepsilon}\|^{2}_{L^{2^{\ast\ast}}(\R^N)}+\| \Delta{\varphi}_{\varepsilon}\|^{2}_{L^{2}({\R^N})}\right)\\
&\leq C_{1}\left( \|\Delta{\varphi}_{\varepsilon}\|^{2}_{L^{2}(\R^N)}+\| \Delta{\varphi}_{\varepsilon}\|^{2}_{L^{2}({\R^N})}\right)\\
&\leq C_{1}\left( \mathop{Cap_{\Delta}}_{\R^N}( Z)+\varepsilon\right),
\end{aligned}
\end{equation*}
where the last but one inequality using the Sobolev embedding theorem.
Letting $\varepsilon\rightarrow0$, we have 	\begin{equation*}
\mathop{Cap_{\Delta}}_{\Omega}( Z)\leq C_{1}	\mathop{Cap_{\Delta}}_{\R^N}( Z).
\end{equation*}
The proof is complete.
\end{proof}
Since $ Z$ is compact and 	${\mathop{Cap_{\Delta}}\limits_{\R^N}}( Z)=0$, we know that $ Z_{\lambda}$ is compact and 	${\mathop{Cap_{\Delta}}\limits_{\R^N}}( Z_{\lambda})=0$. Moreover, from Lemma \ref{cap0}, we obtain $\mathop{Cap_{\Delta}}\limits_{B_{\varepsilon}^{\lambda}}( Z_{\lambda})=0$ for any open neighborhood $B_{\varepsilon}^{\lambda}$ of $ Z_{\lambda}$, where $B_{\varepsilon}^{\lambda}:=\left\lbrace x\in\R^N\ | \ dist(x, Z_{\lambda})<\varepsilon\right\rbrace $ is the tubular neighborhood of radius $\varepsilon$ and centered about $ Z_{\lambda}$, see \cite{PPS}. Thus, by the definition of \eqref{defcap}, for any $\varepsilon>0$, there exists $\varphi_{\varepsilon}\in C^{\infty}_{c}(B_{\varepsilon}^{\lambda})$,  $\varphi\geq1$ in a neighborhood $B_{\delta}^{\lambda}\subset B_{\varepsilon}^{\lambda}$ of $ Z_{\lambda}$, such that
\begin{equation}\label{rn}
 \int_{B_{\varepsilon}^{\lambda}}|\Delta\varphi_{\varepsilon}|^{2}dx<\varepsilon .
\end{equation}
Then we can construct a function $\eta_{\varepsilon}\in C^{2}(\R^N)$, which satisfies
\begin{equation}\label{eta}
\eta_{\varepsilon}=\begin{cases} 0 & \text { in }  B_{\delta}^{\lambda} \\ 1 & \text { in } \R^N\setminus B_{\varepsilon}^{\lambda}\end{cases},
\end{equation}
and
\begin{equation}\label{eta1}
	\int_{B_{\varepsilon}^{\lambda}}|\Delta \eta_{\varepsilon}|^{2}dx+	\int_{B_{\varepsilon}^{\lambda}}|\nabla \eta_{\varepsilon}|^{2}dx<C\varepsilon,
\end{equation}
where the constant $C>0$ is independing on $\varepsilon$.
To this end we define a function $h(s)\in C^{2}(\R)$, that is
\begin{equation*}
h(t)=
\begin{cases} 1 & \text { if }  t\leq 0 \\ -8(t-\frac{1}{2})^{3}(24t^{2}+6t+1) & \text { if } 0<t<\frac{1}{2}\\ 0 & \text { if }  t\geq \frac{1}{2}
\end{cases}.
\end{equation*}
Let $\eta_{\varepsilon}=h(\varphi_{\varepsilon})$, we have extended $\varphi_{\varepsilon}=0$ in $\R^N\setminus B_{\varepsilon}^{\lambda}$, and it follows that $\eta_{\varepsilon}\in C^{2}(\R^N)$, $0\leq\eta_{\varepsilon}\leq1$.
In addition, there holds
\begin{equation*}
	\int_{B_{\varepsilon}^{\lambda}}|\Delta \eta_{\varepsilon}|^{2}dx\leq C\int_{B_{\varepsilon}^{\lambda}}\left( |\nabla \varphi_{\varepsilon}|^{2}+|\Delta \varphi_{\varepsilon}|\right)^{2}dx\leq C\int_{B_{\varepsilon}^{\lambda}}\left( |\nabla \varphi_{\varepsilon}|^{4}+|\Delta \varphi_{\varepsilon}|^{2}\right)dx\leq C\int_{B_{\varepsilon}^{\lambda}}|\Delta \varphi_{\varepsilon}|^{2}dx<C\varepsilon,
\end{equation*}
the third inequality by using the Gagliardo-Nirenberg inequality, see \cite{NL}.
By the H\"{o}lder inequality, we have
\begin{equation*}
	\begin{aligned}
\int_{B_{\varepsilon}^{\lambda}}|\nabla \eta_{\varepsilon}|^{2}dx
&\leq C\int_{B_{\varepsilon}^{\lambda}}|\nabla \varphi_{\varepsilon}|^{2}dx
\leq C\left(  \mathcal{L}^{N}(B_{\varepsilon}^{\lambda})\right)^{\frac{1}{2}} \left( \int_{B_{\varepsilon}^{\lambda}}|\nabla \varphi_{\varepsilon}|^{4}dx\right) ^{\frac{1}{2}}\\
&\leq C\left(  \mathcal{L}^{N}(B_{\varepsilon}^{\lambda})\right)^{\frac{1}{2}} \left(  \int_{B_{\varepsilon}^{\lambda}}|\Delta \varphi_{\varepsilon}|^{2}dx\right) ^{\frac{1}{2}} <C\varepsilon,
	\end{aligned}
\end{equation*}
where $\mathcal{L}^{N}(\cdot)$ is the Lebesgue measure.

\medskip

Now we deduce some crucial estimates for the method of moving planes.
\begin{lem}\label{lem1}
Under the assumption of Theorem \ref{symm}, then for any $\lambda<0$, there holds
\begin{equation}\label{w1}
	\|w_{\lambda}^{+}\|_{L^{2^{\ast\ast}}(\Sigma_{\lambda}^{\ast})}\leq 	\|(-\Delta w_{\lambda})^{+}\|_{L^{2}(\Sigma_{\lambda}^{\ast})}
\end{equation}	
and
\begin{equation}\label{w2}
		\|(-\Delta w_{\lambda})^{+}\|_{L^{2}(\Sigma_{\lambda}^{\ast})}\leq \|u\|_{L^{2^{\ast\ast}}(\Sigma_{\lambda}^{\ast,+})}^{\frac{8}{N-4}}	\| w_{\lambda}^{+}\|_{L^{2^{\ast\ast}}(\Sigma_{\lambda}^{\ast})}	,
\end{equation}
where $w_{\lambda}^{+}=\max\left\lbrace 0, w_{\lambda} \right\rbrace $ and $\Sigma_{\lambda}^{\ast,+}=\left\lbrace x\in\Sigma_{\lambda}^{\ast}\ |\ w_{\lambda}^{+}>0 \right\rbrace $.
\end{lem}
\begin{proof}
	 For any functions $f\in L^{2}(\Sigma_{\lambda})$, we define
	\begin{equation}\label{f2}
		\psi(x):=c_{N}\int_{\Sigma_{\lambda}}\left( \frac{1}{|x-y|^{N-2}}-\frac{1}{|x-y^{\lambda}|^{N-2}}\right) f(y)dy, \sp x\in\Sigma_{\lambda}
	\end{equation}
	with $c_{N}=((N-2)|\mathbb{S}^{N-1}|)^{-1}$, then $\psi\in H^{2}(\Sigma_{\lambda})\cap W^{1,\frac{2N}{N-2}}_{0}(\Sigma_{\lambda})$ and $-\Delta\psi=f$ in $\Sigma_{\lambda}$, where $ H^{2}(\Sigma_{\lambda})$ and $ W^{1,\frac{2N}{N-2}}_{0}(\Sigma_{\lambda})$ is the completion of $C_{c}^{\infty}(\Sigma_{\lambda})$ with respect to $\|\Delta\psi\|_{L^{2}(\Sigma_{\lambda})}$ and $\|\nabla\psi\|_{L^{\frac{2N}{N-2}}(\Sigma_{\lambda})}$. Using the Hardy-Littlewood-Sobolev inequality to \eqref{f2}, we get
	\begin{equation}\label{hls}
		\|\psi\|_{L^{r}(\Sigma_{\lambda})}\leq C(r,N)\|f\|_{L^{s}(\Sigma_{\lambda})},
	\end{equation}
where $1<s<r<\infty$ and $1+\frac{1}{r}=\frac{N-2}{N}+\frac{1}{s}$.
	First, we prove the inequality \eqref{w1}. Denote
	\begin{equation}\label{f3}
		f_k(x)=\begin{cases}(w_{\lambda}^{+})^{\frac{N+4}{N-4}}(x) \chi_{\left\{w_{\lambda}^{+} \leq k\right\}}(x) \chi_{\left\{\operatorname{dist}(x,  Z_{\lambda}) \geq k^{-1},|x| \leq k\right\}}(x), & x\in\Sigma_{\lambda}^{\ast}\\
			0, & x\in Z_{\lambda}\end{cases}
	\end{equation}
	$k\in\mathbb{N}$, then $f_{k}$ is nonnegative, bounded and  has compact support, which implies $f_{k}\in L^{2}(\Sigma_{\lambda})$. Thus,
	the corresponding $\psi_{k}$ related to $f_{k}$, which is defined by \eqref{f2} satisfies $-\Delta\psi_{k}=f_{k}=0$ in the neighborhood $B_{\frac{1}{k}}^{\lambda}$ of $ Z_{\lambda}$.
	Since $f_{k}$ is nondecreasing with respect to $k$, it follows by the monotone convergence theorem that
	\begin{equation}\label{w0}
		\int_{\Sigma_{\lambda}^{\ast}} (w_{\lambda}^{+})^{\frac{2 N}{N-4}}dx=	\int_{\Sigma_{\lambda}^{\ast}} w_{\lambda}(w_{\lambda}^{+})^{\frac{ N+4}{N-4}}dx=\lim _{k \rightarrow \infty} \int_{\Sigma_\lambda^{*}} w_{\lambda} f_kdx=\lim _{k \rightarrow \infty} \int_{\Sigma_\lambda^{*}} w_{\lambda}\left(-\Delta \psi_k\right)dx .
	\end{equation}

\medskip
	
	Next we apply the nonnegative cutoff functions $\eta_{\varepsilon}$ as \eqref{eta} to estimate the integral $\int_{\Sigma^*_\lambda} w_{\lambda}\left(-\Delta \psi_k\right)dx$. Claim that
	\begin{equation}\label{w3}
		\lim_{\varepsilon \rightarrow0}\int_{\Sigma_\lambda^{*}} w_{\lambda}\left(-\Delta\left(\psi_k \eta_{\varepsilon}\right)\right)dx=\int_{\Sigma_\lambda^{*}}w_{\lambda}\left(-\Delta\psi_k\right)dx.
	\end{equation}
	For fixed $k$, let $\varepsilon>0$ so small that $B_{\varepsilon}^{\lambda}\subset B_{\frac{1}{k}}^{\lambda}$. It is noted that $-\Delta\psi_{k}=0$ in $B_{\frac{1}{k}}^{\lambda}$, and the boundary of $B_{\frac{1}{k}}^{\lambda}$ defined by $\partial B_{\frac{1}{k}}^{\lambda}:=\left\lbrace x\in\R^N\ | \ \operatorname{dist}(x, Z_{\lambda})=\frac{1}{k}\right\rbrace $ is a smooth embedded hypersurface in $\R^N$, then $\psi_{k}\in C^{\infty}_{loc}(B_{\frac{1}{k}}^{\lambda})$ by harmonicity. Therefore, we get that $\psi_{k}$ and $\nabla\psi_{k}$ are bounded in $B_{\varepsilon}^{\lambda}$. By direct calculation, we have
	\begin{equation}\label{w7}
		\int_{\Sigma_\lambda^{*}} w_{\lambda}\Delta\left(\psi_k \eta_{\varepsilon}\right)dx=\int_{\Sigma_\lambda^{*}}w_{\lambda} \eta_{\varepsilon}\Delta \psi_kdx
		+ \int_{\Sigma_\lambda^{*}}w_{\lambda} \psi_k\Delta\eta_{\varepsilon}dx
		+\int_{\Sigma_\lambda^{*}}2w_{\lambda} \nabla\eta_{\varepsilon} \nabla\psi_kdx.
	\end{equation}
	For the first term on the right side, since $-\Delta\psi_{k}=0$ in $B_{\varepsilon}^{\lambda}$ and $\eta_{\varepsilon}=1$ in $\R^N\setminus B_{\varepsilon}^{\lambda}$,  we have
	\begin{equation*}
		\int_{\Sigma_\lambda^{*}}w_{\lambda} \eta_{\varepsilon}\Delta \psi_kdx=	\int_{\Sigma_\lambda^{*}}w_{\lambda} \Delta \psi_kdx.
	\end{equation*}
	For the latter two terms on the right side, by using the H\"{o}lder inequality and the boundedness of $\psi_{k}$ and $\nabla\psi_{k}$, we obtain
	\begin{equation*}
		\left|\int_{\Sigma_\lambda^{*}}w_{\lambda} \psi_k\Delta\eta_{\varepsilon}dx\right|
		\leq C\int_{B_{\varepsilon}^{\lambda}\setminus B_{\delta}^{\lambda}}|w_{\lambda}|| \Delta\eta_{\varepsilon}|dx
		\leq C\left(  \mathcal{L}^{N}(B_{\varepsilon}^{\lambda})\right)^{\frac{4}{2N}}\|w_{\lambda}\|_{L^{2^{\ast\ast}}(B_{\varepsilon}^{\lambda}\setminus B_{\delta}^{\lambda})} \|\Delta\eta_{\varepsilon}\|_{L^{2}(B_{\varepsilon}^{\lambda}\setminus B_{\delta}^{\lambda})}
	\end{equation*}
	and
	\begin{equation*}
		\left|\int_{\Sigma_\lambda^{*}}w_{\lambda} \nabla\eta_{\varepsilon} \nabla\psi_kdx\right|\leq C\int_{B_{\varepsilon}^{\lambda}\setminus B_{\delta}^{\lambda}}|w_{\lambda}|| \nabla\eta_{\varepsilon}|dx
		\leq C\left(  \mathcal{L}^{N}(B_{\varepsilon}^{\lambda})\right)^{\frac{4}{2N}}\|w_{\lambda}\|_{L^{2^{\ast\ast}}(B_{\varepsilon}^{\lambda}\setminus B_{\delta}^{\lambda})} \|\nabla\eta_{\varepsilon}\|_{L^{2}(B_{\varepsilon}^{\lambda}\setminus B_{\delta}^{\lambda})}.
	\end{equation*}
	Therefore, combining the above three integrations and taking $\varepsilon\rightarrow 0$ in \eqref{w7}, this means that \eqref{w3} holds.

\medskip

	Since $\psi_{k}=0$ and $w_{\lambda}=0$ on $T_{\lambda}$, integrating by parts and using the trace theorem, we get
	\begin{equation}\label{w4}
		\lim_{\varepsilon\rightarrow0}\int_{\Sigma_\lambda^{*}} w_{\lambda}\left(-\Delta\left(\psi_k \eta_{\varepsilon}\right)\right)dx=\lim_{\varepsilon\rightarrow0}\int_{\Sigma_\lambda^{*}}(-\Delta w_{\lambda})\left(\psi_k \eta_{\varepsilon}\right)dx \leq\int_{\Sigma_\lambda^{*}}(-\Delta w_{\lambda})^{+}\psi_k dx.
	\end{equation}
	Then  by \eqref{w3} and \eqref{w4}, it follows that
	\begin{equation}\label{w5}
		\begin{aligned}
			\int_{\Sigma_\lambda^{*}} w_{\lambda}\left(-\Delta\psi_k\right)dx &\leq\int_{\Sigma_\lambda^{\ast}}(-\Delta w_{\lambda})^{+}\psi_k dx\\
			&\leq	\|(-\Delta w_{\lambda})^{+}\|_{L^{2}(\Sigma_{\lambda}^{\ast})}\|\psi_{k}\|_{L^{2}(\Sigma_{\lambda}^{\ast})}\\
			&\leq 	\|(-\Delta w_{\lambda})^{+}\|_{L^{2}(\Sigma_{\lambda}^{\ast})}\|f_{k}\|_{L^{\frac{2N}{N+4}}(\Sigma_{\lambda}^{\ast})}.
		\end{aligned}
	\end{equation}
	The second inequality holds by using the H\"older inequality and the third one is \eqref{hls} with $s=\frac{2N}{N+4}$ and $r=2$. Combining with \eqref{w0}, we have
	\begin{equation}
		\int_{\Sigma_{\lambda}^{\ast}} (w_{\lambda}^{+})^{\frac{2 N}{N-4}}dx \leq  \|(-\Delta w_{\lambda})^{+}\|_{L^{2}(\Sigma_{\lambda}^{\ast})}\|w_{\lambda}^{+}\|_{L^{\frac{2N}{N-4}}(\Sigma_{\lambda}^{\ast})}^{\frac{N+4}{N-4}},
	\end{equation}
	as $k\rightarrow\infty$ in \eqref{w5}. Hence, \eqref{w1} holds, since $\|w_{\lambda}^{+}\|_{L^{\frac{2N}{N-4}}(\Sigma_{\lambda}^{\ast})}<\infty$.

\medskip
	
	By a similar argument as the proof of \eqref{w1}, for $k\in\mathbb{N}$,  set
	\begin{equation}\label{f4}
		f_k(x)=\begin{cases} (-\Delta w_{\lambda})^{+}(x) \chi_{\left\{(-\Delta w_{\lambda})^{+} \leq k\right\}}(x) \chi_{\left\{\operatorname{dist}(x,  Z_{\lambda}) \geq k^{-1},|x| \leq k\right\}}(x), & x\in\Sigma_{\lambda}^{\ast}\\
			0, & x\in Z_{\lambda}\end{cases}
	\end{equation}
	which implies $f_{k}$ is nonnegative, bounded and  has compact support, $f_{k}\in L^{2}(\Sigma_{\lambda})$. Thus,
	we also have that the corresponding $\psi_{k}$ related to $f_{k}$, which is defined by \eqref{f2} verfies $-\Delta\psi_{k}=f_{k}=0$ in the neighborhood $B_{\frac{1}{k}}^{\lambda}$ of $ Z_{\lambda}$.
	Since $f_{k}$ is nondecreasing with respect to $k$, by the monotone convergence theorem, we get
	\begin{equation}\label{w6}
		\int_{\Sigma_{\lambda}^{\ast}}  \left( (-\Delta w_{\lambda})^{+}\right) ^{2}dx=	\int_{\Sigma_{\lambda}^{\ast}} (-\Delta w_{\lambda})(-\Delta w_{\lambda})^{+}dx=\lim _{k \rightarrow \infty} \int_{\Sigma_\lambda^{*}} (-\Delta w_{\lambda}) f_kdx=\lim _{k \rightarrow \infty} \int_{\Sigma_\lambda^{*}} \Delta w_{\lambda}\Delta \psi_kdx .
	\end{equation}
As $\psi_{k}\in H^{2}(\Sigma_{\lambda})\cap W^{1,\frac{2N}{N-2}}_{0}(\Sigma_{\lambda})$, by density arguments, we plug $\varphi:=\psi_{k}\eta_{\varepsilon}\chi_{\Sigma_{\lambda}}$ as test function in \eqref{u0} and in \eqref{u1}, then subtracting, we get
	\begin{equation}\label{u3}
		\int_{\Sigma_{\lambda}^{*}}\Delta w_{\lambda}  \Delta\left(\psi_{k}\eta_{\varepsilon}\right)  dx=	\int_{\Sigma_{\lambda}^{*}}\left( u^{\frac{N+4}{N-4}}-u_{\lambda}^{\frac{N+4}{N-4}} \right) \psi_{k}\eta_{\varepsilon} dx.
	\end{equation}
	Arguing as \eqref{w7}, by the H\"older inequality, the left side of the above equality is that \begin{equation*}
		\begin{aligned}
			\int_{\Sigma_\lambda^{*}} \Delta w_{\lambda}\Delta\left(\psi_k \eta_{\varepsilon}\right)dx
			=&\int_{\Sigma_\lambda^{*}}(\Delta w_{\lambda}) \eta_{\varepsilon}\Delta \psi_kdx
			+ \int_{\Sigma_\lambda^{*}}(\Delta w_{\lambda}) \psi_k\Delta\eta_{\varepsilon}dx
			+\int_{\Sigma_\lambda^{*}}2\Delta w_{\lambda} \nabla\eta_{\varepsilon} \nabla\psi_kdx.\\
		\end{aligned}
	\end{equation*}
Meanwhile, we have
\begin{equation*}
\int_{\Sigma_\lambda^{*}}(\Delta w_{\lambda}) \psi_k\Delta\eta_{\varepsilon}dx\leq C\|\Delta w_{\lambda}\|_{L^{2}(B_{\varepsilon}^{\lambda}\setminus B_{\delta}^{\lambda})} \|\Delta\eta_{\varepsilon}\|_{L^{2}(B_{\varepsilon}^{\lambda}\setminus B_{\delta}^{\lambda})},
\end{equation*}
and
\begin{equation*}
			\int_{\Sigma_\lambda^{*}}\Delta w_{\lambda} \nabla\eta_{\varepsilon} \nabla\psi_kdx\leq C\|\Delta w_{\lambda}\|_{L^{2}(B_{\varepsilon}^{\lambda}\setminus B_{\delta}^{\lambda})} \|\nabla\eta_{\varepsilon}\|_{L^{2}(B_{\varepsilon}^{\lambda}\setminus B_{\delta}^{\lambda})},
\end{equation*}
	from which we can conclude that
	\begin{equation}\label{eta2}
		\lim\limits_{\varepsilon\rightarrow 0}\int_{\Sigma_{\lambda}^{*}}\Delta w_{\lambda}  \Delta\left(\psi_{k}\eta_{\varepsilon}\right)  dx=	\int_{\Sigma_{\lambda}^{*}}\Delta w_{\lambda}  \Delta\psi_{k}  dx.
	\end{equation}
	 From \eqref{eta2}, letting $\varepsilon\rightarrow 0$ in \eqref{u3}, we have
	\begin{equation*}
		\begin{aligned}
			\int_{\Sigma_{\lambda}^{*}}\Delta w_{\lambda}  \Delta\psi_{k}dx&=	\int_{\Sigma_{\lambda}^{\ast}}\left( u^{\frac{N+4}{N-4}}-u_{\lambda}^{\frac{N+4}{N-4}} \right) \psi_{k}dx\\
			&\leq C \int_{\Sigma_{\lambda}^{\ast}}u^{\frac{8}{N-4}}w_{\lambda}^{+} \psi_{k}dx\\
			&\leq\|u\|_{L^{2^{\ast\ast}}(\Sigma_{\lambda}^{\ast,+})}^{\frac{8}{N-4}}	\| w_{\lambda}^{+}\|_{L^{2^{\ast\ast}}(\Sigma_{\lambda}^{\ast})}\|\psi_{k}\|_{L^{\frac{2N}{N-4}}(\Sigma_{\lambda}^{\ast})}\\
			&\leq \|u\|_{L^{2^{\ast\ast}}(\Sigma_{\lambda}^{\ast,+})}^{\frac{8}{N-4}}	\| w_{\lambda}^{+}\|_{L^{2^{\ast\ast}}(\Sigma_{\lambda}^{\ast})}\|f_{k}\|_{L^{2}(\Sigma_{\lambda}^{\ast})},
		\end{aligned}
	\end{equation*}
	where $\Sigma_{\lambda}^{\ast,+}=\left\lbrace x\in\Sigma_{\lambda}^{\ast}\ |\ w_{\lambda}^{+}>0 \right\rbrace $ and we also use the H\"older inequality and \eqref{hls} with $s=2$ and $r=\frac{2N}{N-4}$. Then from \eqref{w6}, as $k\rightarrow\infty$, we deduce  that
	\begin{equation*}
		\int_{\Sigma_{\lambda}^{\ast}}  \left( (-\Delta w_{\lambda})^{+}\right) ^{2}dx\leq \|u\|_{L^{2^{\ast\ast}}(\Sigma_{\lambda}^{\ast,+})}^{\frac{8}{N-4}}	\| w_{\lambda}^{+}\|_{L^{2^{\ast\ast}}(\Sigma_{\lambda}^{\ast})}\|(-\Delta w_{\lambda})^{+}\|_{L^{2}(\Sigma_{\lambda}^{\ast})}
	\end{equation*}
	which implies \eqref{w2} holds. The proof is complete.
\end{proof}

\medskip

\textbf{Proof of Theorem \ref{symm}.} We will split the proof in two steps.\\
\textit{Step 1.} For $\lambda<0$ sufficiently negative, we have $w_{\lambda}\leq 0$ in $\Sigma_{\lambda}^{\ast}$.

From \eqref{w1} and \eqref{w2}, we easily get that
\begin{equation*}
	\|w_{\lambda}^{+}\|_{L^{2^{\ast\ast}}(\Sigma_{\lambda}^{\ast})}\leq 	\|(-\Delta w_{\lambda})^{+}\|_{L^{2}(\Sigma_{\lambda}^{\ast})}\leq \|u\|_{L^{2^{\ast\ast}}(\Sigma_{\lambda}^{\ast})}^{\frac{8}{N-4}}	\| w_{\lambda}^{+}\|_{L^{2^{\ast\ast}}(\Sigma_{\lambda}^{\ast})}.	
\end{equation*}	
  Due to the integrability of $u(x)$, letting $\lambda\rightarrow-\infty$, we have $\|u\|_{L^{2^{\ast\ast}}(\Sigma_{\lambda}^{\ast})}^{\frac{8}{N-4}}<1$, which implies $	\|w_{\lambda}^{+}\|_{L^{2^{\ast\ast}}(\Sigma_{\lambda}^{\ast})}=0$. Thus, we obtain $w_{\lambda}\leq 0$ in $\Sigma_{\lambda}^{\ast}$. Now we can begin to move the plane $T_{\lambda}$ from $-\infty$ to the right as long as  $w_{\lambda}\leq 0$ in $\Sigma_{\lambda}^{\ast}$. Denote
$$\lambda_{0}=\sup\left\lbrace \lambda< 0\ |  \ w_{\mu}\leq 0 \ \text{in}\ \Sigma_{\mu}^{\ast},  \mu\leq\lambda\right\rbrace. $$
\textit{Step 2.} We prove $\lambda_{0}=0$.

Argue by contradiction and suppose that $\lambda_{0}<0$. Then we have $w_{\lambda_{0}}\leq 0$ in $\Sigma_{\lambda_{0}}^{\ast}$ by continuity.  Claim that $w_{\lambda_{0}}< 0$ in $\Sigma_{\lambda_{0}}^{\ast}$. In fact, if $w_{\lambda_{0}} \equiv0$ in $\Sigma_{\lambda_{0}}^{\ast}$, then $u$ would be singular somewhere on $ Z_{\lambda}$, which is impossible since $\lambda_{0}<0$ and $u$ would be continuous on $Z_{\lambda}$.  By \eqref{w2} and  $w_{\lambda_{0}}^{+}= 0$ in $\Sigma_{\lambda_{0}}^{\ast}$, we know that	$\|(-\Delta w_{\lambda})^{+}\|_{L^{2}(\Sigma_{\lambda}^{\ast})}=0$, then it follows  $-\Delta w_{\lambda_{0}}\leq 0$ in $\Sigma_{\lambda_{0}}^{\ast}$. Notice that $w_{\lambda_{0}} \not\equiv0$ in $\Sigma_{\lambda_{0}}^{\ast}$,  we have
\begin{equation}\label{ww0}
	 w_{\lambda_{0}}< 0\sp\text{in}~\Sigma_{\lambda_{0}}^{\ast},
\end{equation}  thanks to the strong maximum principle. Next we will show that there exists a $\tau>0$ such that $ w_{\lambda}\leq  0$ in $\Sigma_{\lambda}^{\ast}$, for any $\lambda\in\left[ \lambda_{0}, \lambda_{0}+\tau\right)$.
Since $u\in L^{2^{\ast\ast}}(\Sigma_{\lambda_{0}})$, for any $\sigma>0$, there exist $\tau_{1}=\tau_{1}(\sigma,\lambda_{0})>0$ such that $\lambda_{0}+\tau_{1}<0$ and a compact set  $K\subset\Sigma_{\lambda}^{\ast}$, here $\lambda\in\left[ \lambda_{0},\lambda_{0}+\tau_{1}\right]$, such that
\begin{equation}\label{ww2}
		\int_{\Sigma_{\lambda_{0}}^{\ast}\setminus K} u^{2^{\ast\ast}}dx<\frac{\sigma}{2}.
\end{equation}
By the continuity of $w_{\lambda}(x)$ on $K\times \left[ \lambda_{0},\lambda_{0}+\tau_{1}\right]$ and \eqref{ww0}, there exist $\tau_{2}\in(0,\tau_{1})$ such that $w_{\lambda}<0$ in $K\subset\Sigma_{\lambda}^{\ast}$, for any $\lambda\in\left[ \lambda_{0},\lambda_{0}+\tau_{2}\right] $.
 By \eqref{ww2} and $u\in L^{2^{\ast\ast}}(\Sigma_{\lambda_{0}+\tau_{2}})$, we can choose $\tau\in(0,\tau_{2})$ such that
\begin{equation*}
	\int_{\Sigma_{\lambda}^{\ast}\setminus K} u^{2^{\ast\ast}}dx<\sigma,
\end{equation*}
for every $\lambda\in\left[ \lambda_{0},\lambda_{0}+\tau\right)$. We observe that $w_{\lambda}^{+}=0$ in $K$ for any $\lambda\in\left[ \lambda_{0},\lambda_{0}+\tau\right)$, together with Lemma \ref{lem1}, there holds
\begin{equation}\label{ww1}
	\|w_{\lambda}^{+}\|_{L^{2^{\ast\ast}}(\Sigma_{\lambda}^{\ast}\setminus K)}\leq \|u\|_{L^{2^{\ast\ast}}(\Sigma_{\lambda}^{\ast}\setminus K)}^{\frac{8}{N-4}}	\| w_{\lambda}^{+}\|_{L^{2^{\ast\ast}}(\Sigma_{\lambda}^{\ast}\setminus K)}.	
\end{equation}	
We choose $\sigma>0$ sufficiently small such that
\begin{equation*}
	\|u\|_{L^{2^{\ast\ast}}(\Sigma_{\lambda}^{\ast}\setminus K)}^{\frac{8}{N-4}}<1,
\end{equation*}
by \eqref{ww1}, it follows $ w_{\lambda}\leq  0$ in $\Sigma_{\lambda}^{\ast}$, for any $\lambda\in\left[ \lambda_{0}, \lambda_{0}+\tau\right)$. This is a contradiction
to the definition of $\lambda_{0}$.
Therefore, we obtain the desired symmetry of $u$. The proof is complete.
$\hfill{} \Box$

\medskip

The proof of Theorem \ref{symm2} follows from Theorem \ref{symm}.

\textbf{Proof of Theorem \ref{symm2}.}
 It suffices to show the symmetry of $u$ under the assumption $z_{0}=0$, that is $0\notin Z$ and
\begin{equation}\label{capk0}
\mathop{Cap_{\Delta}}_{\R^N}(\mathcal{K}(Z))=0,
\end{equation}
where $\mathcal{K}(Z)=\left\lbrace \frac{x}{|x|^{2}},\ x\in Z\right\rbrace $. Indeed, fixing $0\neq z_0\in \left\lbrace x_{1}=0\right\rbrace\setminus Z$, if $u$ is a positive weak solution of \eqref{main}, then $u_{z_0}(x):=u(x+z_0)$ is a solution of 	\begin{equation*}
	\Delta^{2} u=u^{\frac{N+4}{N-4}}
	\quad
	\quad \text{in} \ \mathbb{R}^{N}\setminus \left\lbrace -z_{0}\right\rbrace +Z.
\end{equation*}
We observe that $0\notin -z_{0}+ Z$ and $-z_{0}+ Z$ is a closed and proper subset of the hyperplane $ \left\lbrace x_{1}=0\right\rbrace $ satisfies $$\mathop{Cap_{\Delta}}_{\R^N}(\mathcal{K}(-z_{0}+ Z))=0.$$
Thus, without loss of generality, we assume that $z_{0}=0$ and $u$ is a positive weak solution of \eqref{main}. Now consider the $C^{2}$-diffeomorphism map $\mathcal{K}:\R^N\setminus\left\lbrace 0\right\rbrace \rightarrow \R^N\setminus\left\lbrace 0\right\rbrace$ defined by $\mathcal{K}(x):=\frac{x}{|x|^{2}}$,
then the Kelvin transformation for the solution $u$
$$v(y)=|x|^{N-4}u(x),\quad y=\frac{x}{|x|^{2}},\quad x\in \R^{N}\setminus \left\lbrace { Z\cup\left\lbrace0 \right\rbrace }\right\rbrace  $$ weakly satisfies the equation
	\begin{equation}\label{main2}
	\Delta^{2} v=v^{\frac{N+4}{N-4}}
	\quad
	\quad \text{in} \ \mathbb{R}^{N}\setminus \left\lbrace {\mathcal{K}( Z)\cup\left\lbrace0 \right\rbrace }\right\rbrace.
\end{equation}
Since $0\notin Z$ and $ Z$ is a proper subset of the hyperplane $ \left\lbrace x_{1}=0\right\rbrace $, then we see that $\mathcal{K}( Z)\subset\left\lbrace x_{1}=0\right\rbrace$ is a bounded subset. Moreover, we have $\mathcal{K}( Z)$ is a compact set since $ Z$ is closed and $\mathcal{K}$ is $C^{2}$-diffeomorphism.  For simplicity, we denote $ Z_{0}:=\left\lbrace {\mathcal{K}( Z)\cup\left\lbrace0 \right\rbrace }\right\rbrace$, by \eqref{capk0}, then $ Z_{0}$ is a compact subset of the hyperplane $ \left\lbrace x_{1}=0\right\rbrace $ with $\mathop{Cap_{\Delta}}\limits_{\R^N}(Z_{0})=0$.

In order to apply the method of moving plane to $v$, we prove that $v$ satisfies assumptions of Theorem 2.1.
 \begin{lem}\label{v}
 	 Assume that $u\in H^{2}_{loc}(\R^{N}\setminus  Z)$, then we have $v\in H^{2}_{loc}(\R^{N}\setminus  Z_{0})$. Moreover, there holds $v\in L^{2^{\ast\ast}}(\Sigma)$ with $2^{\ast\ast}=\frac{2N}{N-4}$ and $\Delta v\in L^{2}(\Sigma)$ for every $\Sigma$ which is a positive distance away from $ Z_{0}$.
 \end{lem}
\begin{proof}
 For every $\Sigma$ which is a positive distance away from $ Z_{0}$ and  $v(y)=\frac{1}{|y|^{N-4}}u(\frac{y}{|y|^{2}})$, $ y\in \Sigma $, we have
 	\begin{equation*}
		\int_{\Sigma}v^{2^{\ast\ast}}(y)dy=\int_{\mathcal{K}(\Sigma)}u^{2^{\ast\ast}}(x)dx<\infty,
	\end{equation*}
by Sobolev embedding since $\mathcal{K}(\Sigma)$ is bounded and a positive distance away from $Z$. Direct calculation shows that
\begin{equation}\label{vk}
	\Delta v(y)=-2(N-4)|y|^{2-N}u\left( \frac{y}{|y|^{2}}\right) -4|y|^{-N}y\cdot\nabla u\left( \frac{y}{|y|^{2}}\right)+|y|^{-N}\Delta u\left( \frac{y}{|y|^{2}}\right).
\end{equation}
Then we have
	\begin{equation*}
	\int_{\Sigma}\left||y|^{2-N}u\left( \frac{y}{|y|^{2}}\right) \right| ^{2}dy=\int_{\mathcal{K}(\Sigma)}|x|^{-4}\left|u\left( x\right) \right| ^{2}dx\leq\int_{\widetilde{\mathcal{K}(\Sigma)}}\left|\Delta u\left( x\right) \right| ^{2}dx<\infty,
\end{equation*}
using the Hardy-Rellich inequality and a basic cut-off function argument, where $\widetilde{\mathcal{K}(\Sigma)}$ is also a positive distance away from $ Z_{0}$ and $\overline{\mathcal{K}(\Sigma)}\subset\widetilde{\mathcal{K}(\Sigma)}$.
Similarly, we have
	\begin{equation*}
		\begin{aligned}
		\int_{\Sigma}\left||y|^{-N}y\cdot \nabla u\left( \frac{y}{|y|^{2}}\right) \right| ^{2}dy
		&\leq	\int_{\Sigma}|y|^{2-2N}\left|\nabla u\left( \frac{y}{|y|^{2}}\right) \right| ^{2}dy\\
		&=\int_{\mathcal{K}(\Sigma)}|x|^{-2}\left|\nabla u\left( x\right) \right| ^{2}dx\\
		&\leq\int_{\widetilde{\mathcal{K}(\Sigma)}}\left|\Delta u\left( x\right) \right| ^{2}dx<\infty.
		\end{aligned}
\end{equation*}

Meanwhile, we also get
	\begin{equation*}
	\int_{\Sigma}|y|^{-2N}\left|\Delta u\left( \frac{y}{|y|^{2}}\right) \right| ^{2}dy=\int_{\mathcal{K}(\Sigma)}\left|\Delta u\left( x\right) \right| ^{2}dx<\infty.
\end{equation*}
Combining \eqref{vk} with these integrals over $\Sigma$, we easily obtain that
	\begin{equation*}
	\int_{\Sigma}\left|\Delta v\left( y\right) \right| ^{2}dy<\infty.
\end{equation*}
\end{proof}
Therefore, by Lemma \ref{v}, we know that the Kelvin transformation $v$ satisfies the assumptions in Theorem \ref{symm}, and $ Z_{0}$ is a compact subset of the hyperplane $ \left\lbrace x_{1}=0\right\rbrace $ with ${\mathop{Cap_{\Delta}}\limits_{\R^N}}( Z_{0})=0$, then it follows that $v$ is symmetric with respect to the  hyperplane $ \left\lbrace x_{1}=0\right\rbrace $. Thus, we obtain the symmetry of $u$.

\medskip

Furthermore, if $Z$ is compact subset, without loss of generality, we also can assume that $0\notin Z$. Indeed, analogous arguments as above, for fixed $z_0\in \left\lbrace x_{1}=0\right\rbrace\setminus Z$, then $u_{z_0}(x):=u(x+z_0)$ is a solution of 	\begin{equation*}
	\Delta^{2} u=u^{\frac{N+4}{N-4}}
	\quad
	\quad \text{in} \ \mathbb{R}^{N}\setminus \left\lbrace -z_{0}+ Z\right\rbrace.
\end{equation*}
In addition, $0\notin -z_{0}+ Z$ and $-z_{0}+ Z$ is a compact subset of the hyperplane $ \left\lbrace x_{1}=0\right\rbrace $ satisfies $$\mathop{Cap_{\Delta}}_{\R^N}(-z_{0}+ Z)=0,$$ by \eqref{capacity0}. We only need to show that 	${\mathop{Cap_{\Delta}}\limits_{\R^N}}(\mathcal{K}( Z))=0$ by the following two lemmas.

\begin{lem}\label{cap1}
	Let $F:\R^N\setminus\left\lbrace 0\right\rbrace \rightarrow \R^N\setminus\left\lbrace 0\right\rbrace$ be a $C^{2}$-diffeomorphism and $A$ be a bounded open set of $\R^N\setminus\left\lbrace 0\right\rbrace$ with Lipschitz boundary. If $C\subset A$  is a compact set such that
	\begin{equation}\label{ca1}
		{\mathop{Cap_{\Delta}}\limits_{A}}(C)=0,
	\end{equation}
	then
	\begin{equation}\label{ca2}
		{\mathop{Cap_{\Delta}}\limits_{F(A)}}(F(C))=0.
	\end{equation}
\end{lem}
\begin{proof}
 By \eqref{ca1} and the definition of \eqref{defcap},  for any $\varepsilon>0$, there exists $\varphi_{\varepsilon}\in C^{\infty}_{c}(A)$,  $\varphi_{\varepsilon}\geq1$ in a neighborhood $ B_{\varepsilon}$ of $C$, such that
	\begin{equation*}
		\int_{A}|\Delta\varphi_{\varepsilon}(x)|^{2}dx<\varepsilon .
	\end{equation*}
	Let $G:=F^{-1}$ and $\phi_{\varepsilon}:=\varphi_{\varepsilon}\circ G$, then we have $\phi_{\varepsilon}\geq1$ in a neighborhood $ B_{\varepsilon}^{\prime}$ of $F(C)$. For any $y\in F(C)$, denote $x:=G(y)$, then
	\begin{equation*}
		\begin{aligned}
			|\Delta\phi_{\varepsilon}(y)|&=\left|\sum_{i,j,k=1}^{N}\frac{\partial^{2}\varphi_{\varepsilon}}{\partial x_{j}\partial x_{k}}\frac{\partial x_{k}}{\partial y_{i}}\frac{\partial x_{j}}{\partial y_{i}}
		+\sum_{i,j=1}^{N}\frac{\partial\varphi_{\varepsilon}}{\partial x_{j}}\frac{\partial^{2} x_{j}}{\partial y_{i}^{2}}\right| \\
			&\leq C\left(  \|J_{G}\|_{\infty,\overline{F(A)}}^{2}|\nabla^{2}\varphi_{\varepsilon}(G(y))|+\|H_{G}\|_{\infty,\overline{F(A)}}|\nabla\varphi_{\varepsilon}(G(y))|\right) \\
			&\leq C(F,A)\left(  |\nabla^{2}\varphi_{\varepsilon}(G(y))|+|\nabla\varphi_{\varepsilon}(G(y))|\right), 	\end{aligned}
	\end{equation*}
 where $J_{G}$ and $H_{G}$ are the Jacobian matrix and Hessian matrix of $G(y)$ respectively. On the one hand, the norm equivalence conclusion in \cite{A92} implies that
$$\|\nabla^{2}\varphi_{\varepsilon}(G(y))\|_{L^{2}(F(A))}\leq C\|\Delta\varphi_{\varepsilon}(G(y))\|_{L^{2}(F(A))},$$
for some constants $C$ independent of $\varepsilon$. Hence, by the H\"{o}lder inequality and Gagliardo-Nirenberg inequality, we get
	\begin{equation*}
		\begin{aligned}
			\int_{F(A)}|\Delta\phi_{\varepsilon}(y)|^{2}dy
			&\leq \overline{C}(F,A)^{2}\int_{F(A)} |\Delta\varphi_{\varepsilon}(G(y))|^{2}+|\nabla\varphi_{\varepsilon}(G(y))|^{2}dy\\
			&\leq \overline{C}(F,A)^{2}\int_{A}\left( |\Delta\varphi_{\varepsilon}(x)|^{2}+|\nabla\varphi_{\varepsilon}(x)|^{2}\right) |det(J_{F}(x))|dx\\
			&\leq \tilde{C}(F,A)\int_{A}\left( |\Delta\varphi_{\varepsilon}(x)|^{2}+|\nabla\varphi_{\varepsilon}(x)|^{2}\right) dx\\
			&\leq
			\tilde{C}(F,A) \int_{A}|\Delta\varphi_{\varepsilon}(x)|^{2}dx+\tilde{C}(F,A) \left(\int_{A} |\nabla\varphi_{\varepsilon}(x)|^{2^{\ast}} dx\right)^{{\frac{2}{2^{\ast}}}}\\
			&\leq
			\tilde{C}(F,A) \int_{A}|\Delta\varphi_{\varepsilon}(x)|^{2}dx+\tilde{C}(F,A)\int_{A} |\Delta\varphi_{\varepsilon}(x)|^{2} dx<	\tilde{C}(F,A)\varepsilon,
		\end{aligned}
	\end{equation*}
	where the positive constant $\tilde{C}(F,A)$ is independent of $\varepsilon$. Thus, \eqref{ca2} holds.
\end{proof}

Therefore, for the Klevin transform, we further obtain

\begin{lem}\label{kcap}
	Let $ Z$ be a compact subset of $\R^N$ with $N\geq5$. Suppose $0\notin Z$ and
	\begin{equation}\label{ca3}
		{\mathop{Cap_{\Delta}}\limits_{\R^N}}( Z)=0,
	\end{equation}
	then
	\begin{equation}\label{ca4}
		{\mathop{Cap_{\Delta}}\limits_{\R^N}}(\mathcal{K}( Z))=0.
	\end{equation}
\end{lem}
\begin{proof}
 Notice that $0\in\R^N\setminus Z$ and $ Z$ is a compact subset of $\R^N$, there exist $r_{0}\in(0,1)$ and $R_{0}>0 $ large enough such that $B_{r_{0}}(0)\cap Z=\emptyset$ and  $Z\subset{B_{R_{0}}(0)} \setminus B_{r_0}(0):=A_{0}$. Moreover, since \eqref{ca3}, by Lemma \ref{cap0}, we get
	\begin{equation*}
		{\mathop{Cap_{\Delta}}\limits_{A_{0}}}\left(  Z\right) =0.
	\end{equation*}
 Applying Lemma \ref{cap1}, we have
	\begin{equation*}
		{\mathop{Cap_{\Delta}}\limits_{\mathcal{K}(A_{0})}}\left( \mathcal{K} \left( Z\right)\right)=0,
	\end{equation*}
	which implies ${\rm {\mathop{Cap_{\Delta}}\limits_{\R^N}}\left( \mathcal{K} \left( Z\right)\right)}=0$ by using Lemma \ref{cap0} again.
\end{proof}
Hence, Lemma \ref{kcap} yields that the proof of Theorem 1.2 is complete.

\section{Symmetry result of Punctured solutions}
In this subsection, based on certain appropriate assumptions of the singular set $Z$, we firstly show that every positive punctured solutions of the equation \eqref{main3} satisfies the integral equation \eqref{integ}.
\begin{prop}\label{distribu}
	Let $N\geq 5$. Assume that $u\in C^{4}(\R^{N}\setminus  Z)$ is a positive punctured solution to \eqref{main3} with $p>\frac{N}{N-4}$.
	Suppose that $ Z$ is a compact subset of the hyperplane $ \left\lbrace x_{1}=0\right\rbrace $ with the upper Minkowski dimension $\overline{dim}_{M}( Z)<N-\frac{4p}{p-1}$  or is a smooth k-dimensional closed manifold with $k\leq N-\frac{4p}{p-1}$. Then $u\in L^{p}_{loc}(\R^N)$ and $u$ is a
	distributional solution in $\R^N$, that is,
	\begin{equation}\label{dis}
		\int_{\R^N} u \Delta^{2}\varphi dx=	\int_{\R^N}u^{p} \varphi dx, \quad \forall\varphi\in C^{\infty}_{c}(\mathbb{R}^{N}).
	\end{equation}	
\end{prop}
\begin{proof}
	First, we assume that $ Z$ is a compact subset of the hyperplane $ \left\lbrace x_{1}=0\right\rbrace $ with the upper Minkowski dimension $\overline{\operatorname{dim}}_{M}( Z)<N-\frac{4p}{p-1}$.  Denote that $\mathcal{N}_{\varepsilon}:=\left\lbrace x\in\R^N\ | \ dist(x, Z)<\varepsilon\right\rbrace $ is the tubular neighborhood of radius $\varepsilon$ and centered about $ Z$. For small $\varepsilon>0$, let
	\begin{equation*}
		\xi_{\varepsilon}(x):=1-\int_{\mathcal{N}_{2\varepsilon}}\rho_{\varepsilon}(x-y)dy,
	\end{equation*}
where $\rho\in C^{\infty}_{c}(B_{1})$ with $\int_{B_{1}}\rho dx=1$ and $\rho_{\varepsilon} (x)=\frac{1}{\varepsilon^{N}}\rho(\frac{x}{\varepsilon})$. Then we have that $\xi_{\varepsilon}\in C^{\infty}(\R^N)$ is a non-negative function which satisfies $0\leq\xi_{\varepsilon}(x)\leq1$,  $\xi_{\varepsilon}(x)=0$ on $\mathcal{N}_{\varepsilon}$ and $\xi_{\varepsilon}(x)=1$ on $\mathcal{N}_{3\varepsilon}^{c}$. Moreover, we also have \begin{equation*}
	|\nabla^{j}\xi_{\varepsilon}|\leq C\varepsilon^{-j}, \ \text{for every}\ j=1,2,\cdots
\end{equation*}
where the positive constant $C$ is independent of $\varepsilon$.
Let $R>0$ and $\varphi_{R}\in C^{\infty}_{c}(\R^N)$ be a standard cut-off function such that $\varphi_{R}=1$ on $B_{R}$ and $\varphi_{R}=0$ on $B_{2R}^{c}$. Taking $\phi(x):=\left[(\xi_{\varepsilon}\varphi_{R})(x) \right]^{q} $ with $q=\frac{4p}{p-1}$, by a direct calculation, we get
\begin{equation*}
	\Delta^{2}\phi\leq C(\xi_{\varepsilon}\varphi_{R})^{q-4}+C\varepsilon^{-4}(\xi_{\varepsilon}\varphi_{R})^{q-4}\chi_{\mathcal{N}_{3\varepsilon}\setminus{\mathcal{N}_{\varepsilon}}},
\end{equation*}
since $(\xi_{\varepsilon}\varphi_{R})^{q-k}\leq C (\xi_{\varepsilon}\varphi_{R})^{q-4}$ and $	|\nabla^{k}\xi_{\varepsilon}|\leq C\varepsilon^{-4} $ for each $0\leq k\leq4$,
where  $C>0$ independing on $\varepsilon$. Multiplying both sides of \eqref{main3} by $\phi$ and integrating by parts, we have
\begin{equation}\label{up1}
\begin{aligned}
	\int_{\mathbb{R}^N}u^{p}\phi dx& =\int_{\mathbb{R}^N}u\Delta^2\phi dx \\
	&\leq C\int_{B_{2R}}u\phi^{\frac{q-4}{q}}dx+C\varepsilon^{-4}\int_{B_{2R}\cap(\mathcal{N}_{3\varepsilon}\setminus\mathcal{N}_\varepsilon)}u\phi^{\frac{q-4}{q}}dx \\
	&\leq C\left(1+\varepsilon^{-4}\left( \mathcal{L}^{N}(\mathcal{N}_{3\varepsilon}\setminus\mathcal{N}_\varepsilon)\right) ^{\frac{p-1}{p}}\right)\left(\int_{\mathbb{R}^N}u^{p}\phi dx\right)^{\frac{1}{p}}.
\end{aligned}
\end{equation}
Next we will estimate the Lebesgue measure  $\mathcal{L}^{N}(\mathcal{N}_{3\varepsilon}\setminus\mathcal{N}_\varepsilon)$. We choose $\lambda>\overline{\operatorname{dim}}_M( Z)$ but sufficiently close to $\overline{\operatorname{dim}}_M( Z)$ such that $\frac{(N-\lambda)(p-1)}{p}-4\geq0$, which is equivalent to $\lambda\leq N-\frac{4p}{p-1}$. From Proposition 5.8 in \cite{KLV}, there exist  constants $C>0$ and $r_{0}>0$ such that
 \begin{equation}\label{hn}
 	\mathcal{H}^{N-1}\left(\partial \mathcal{N}_{r}\right) \leq C r^{N-\lambda-1} \quad \text { for all } 0<r<r_0,
 \end{equation}
where $\mathcal{H}^{N-1}$ is the $(N-1)$-dimensional Hausdorff measure on $\R^N$. Since the distance function $d(x)$ to $ Z$ is a $1$-Lipschitz function, by Rademacher's theorem, it is differentiable, a.e. $|\nabla d|=1$. Then by \eqref{hn} and the co-area formula, for $\varepsilon<\frac{r_{0}}{3}$, we obtain
\begin{equation}\label{hn1}
	\begin{aligned}
		\mathcal{L}^{N}(\mathcal{N}_{3\varepsilon}\setminus\mathcal{N}_\varepsilon) \leq \mathcal{L}^{N}\left(\mathcal{N}_{3 \varepsilon}\right) & =\int_0^{3 \varepsilon}\left(\int_{\partial \mathcal{N}_r} 1 d \mathcal{H}^{N-1}\right) d r \\
		& \leq C \int_0^{3 \varepsilon} r^{N-\lambda-1} d r \leq C \varepsilon^{N-\lambda}.
	\end{aligned}
\end{equation}
Inserting \eqref{hn1} into \eqref{up1}, we have
\begin{equation*}
	\int_{\mathbb{R}^N}u^{p}\phi dx
	\leq C\left(1+\varepsilon^{\frac{(N-\lambda)(p-1)}{p}-4}\right) \left(\int_{\mathbb{R}^N}u^{p}\phi dx\right)^{\frac{1}{p}}
	\leq C \left(\int_{\mathbb{R}^N}u^{p}\phi dx\right)^{\frac{1}{p}},
\end{equation*}
which implies
\begin{equation*}
	\int_{B_{R} \cap \mathcal{N}_{3 \varepsilon}^c} u^{p} d x \leq \int_{\mathbb{R}^N} u^{p} \phi dx \leq C .
\end{equation*}
Hence, letting $\varepsilon\rightarrow 0$, we get
\begin{equation*}
	\int_{B_{R} } u^{p} d x  \leq C ,
\end{equation*}
this means $ u \in L_{l o c}^{p}\left(\mathbb{R}^N\right)$.

Now we show that $u$ is a distributional solution in $\R^N$. For any $\varphi \in C_c^{\infty}\left(\mathbb{R}^N\right)$ with $\operatorname{supp} \varphi \subset B_R$, multiplying \eqref{main3} by $\varphi_{\varepsilon}:=\varphi\xi_{\varepsilon}$ and integrating by parts, we have
\begin{equation} \label{up2}
	\int_{\mathbb{R}^N} u^{p} \varphi\xi_{\varepsilon} d x=\int_{\R^N} u \xi_{\varepsilon}\Delta^2 \varphi d x+\int_{B_R} u F_{\varepsilon}(x) d x,
\end{equation}
where $F_{\varepsilon}(x)$ containing the partial derivatives of $\xi_{\varepsilon}$ up to order $4$. By the H\"older inequality, we obtain that
 \begin{equation*}
 	\begin{aligned}
	\left|\int_{B_R} u F_{\varepsilon}(x) d x\right|
	& \leq\left(\int_{B_R \cap\left(\mathcal{N}_{3 \varepsilon} \backslash \mathcal{N}_{\varepsilon}\right)}\left|F_{\varepsilon}\right|^{\frac{p}{p-1}}dx\right)^{\frac{p-1}{p}}\left(\int_{B_R \cap\left(\mathcal{N}_{3 \varepsilon} \backslash \mathcal{N}_{\varepsilon}\right)} u^{p}dx\right)^{\frac{1}{p}} \\
	& \leq C \varepsilon^{-4} \cdot \left( \mathcal{L}^N\left(\mathcal{N}_{3 \varepsilon} \backslash \mathcal{N}_{\varepsilon}\right)\right) ^{\frac{p-1}{p}} \cdot\left(\int_{B_R \cap\left(\mathcal{N}_{3 \varepsilon} \backslash \mathcal{N}_{\varepsilon}\right)} u^{p}dx\right)^{\frac{1}{p}} \\
	 & \leq C \varepsilon^{\frac{(N-\lambda)(p-1)}{p}-4}\left(\int_{B_R \cap\left(\mathcal{N}_{3 \varepsilon} \backslash \mathcal{N}_{\varepsilon}\right)} u^{p}dx\right)^{\frac{1}{p}} \rightarrow 0\end{aligned}
 \end{equation*}
 if $\varepsilon\rightarrow0$, since $ u \in L_{l o c}^{p}\left(\mathbb{R}^N\right)$ and $\frac{(N-\lambda)(p-1)}{p}-4\geq0$. Thus, we see that $u$ is a distributional solution in $\R^N$ as  $\varepsilon\rightarrow0$ in \eqref{up2}.

When $\Gamma$ is a smooth $k$-dimensional closed manifold with $k\leq N-\frac{4p}{p-1}$, the proof is similar to the above, the only difference is that in this case, $\mathcal{L}^{N}\left(\mathcal{N}_{3 \varepsilon}\right) \leq C \varepsilon^{N-k}.$ We complete the proof.
\end{proof}
Furthermore, we will present several crucial lemmas, which are essential for demonstrating the representation formula \eqref{integ}.
\begin{lem}
	Under the assumptions of Proposition \ref{distribu}, then there hold
	\begin{equation}\label{up3}
		\int_{\R^N}\frac{u}{1+|x|^{\gamma}}dx<\infty,~~\text{for}\ \gamma>N-\frac{4}{p-1},
	\end{equation}
and
\begin{equation}\label{up4}
	\int_{\R^N}\frac{u^{p}}{1+|x|^{\gamma}}dx<\infty,~~\text{for}\ \gamma>N-\frac{4p}{p-1}.
\end{equation}
\end{lem}
\begin{proof}
	Let $R>0$ and $\varphi_{R}\in C^{\infty}_{c}(\R^N)$ be a standard cut-off function such that $\varphi_{R}=1$ on $B_{R}$ and $\varphi_{R}=0$ on $B_{2R}^{c}$. Taking $\phi=\varphi_{R} ^{q} $ with $q=\frac{4p}{p-1}$ as a test function in \eqref{dis}, by using of H\"{o}lder inequality, we get
	\begin{equation*}
		\begin{aligned}
			\int_{\mathbb{R}^N}u^{p}\varphi_{R} ^{q} dx& =\int_{\mathbb{R}^N}u\Delta^2(\varphi_{R} ^{q}) dx
			\leq CR^{-4}\int_{B_{2R}}u\varphi_{R} ^{q-4}dx = CR^{-4}\int_{B_{2R}}u\varphi_{R} ^{\frac{q}{p}}dx  \\
			&\leq C R^{\frac{N(p-1)}{p}-4}\left(\int_{\mathbb{R}^N}u^{p}\varphi_{R} ^{q} dx\right)^{\frac{1}{p}}.
		\end{aligned}
	\end{equation*}
	Thus, for every $R>0$, there holds
	\begin{equation}\label{upr}
		\int_{B_R}u^{p}dx\leq CR^{N-\frac{4p}{p-1}},
	\end{equation}
from which we have
\begin{equation}\label{upr1}
	\begin{aligned}
		\int_{\mathbb{R}^N} \frac{u^{p}}{1+|x|^\gamma} d x & =\int_{B_1} \frac{u^{p}}{1+|x|^\gamma} d x+\sum_{i=1}^{\infty} \int_{B_{2^i} \backslash B_{2^{i-1}}} \frac{u^{p}}{1+|x|^\gamma} d x \\
		& \leq C \int_{B_1} u^{p} dx+\sum_{i=1}^{\infty} \int_{B_{2^i} \backslash B_{2^{i-1}}} u^{p} d x \cdot 2^{-\gamma(i-1)} \\
		& \leq C+C \sum_{i=1}^{\infty}\left(2^i\right)^{N-\frac{4p}{p-1}-\gamma}<+\infty ,
	\end{aligned}
\end{equation}
for $\gamma>N-\frac{4p}{p-1}$. Hence, \eqref{up4} holds. Notice that $u\in L^{p}_{loc}(\R^N)$, then there holds $u\in L^{1}_{loc}(\R^N)$. By \eqref{upr} and the  H\"{o}lder inequality, we get
	\begin{equation*}
	\int_{B_R}udx\leq CR^{N-\frac{4}{p-1}}.
\end{equation*}
Taking a similar argument as \eqref{upr1}, we obtain
\begin{equation*}
	\begin{aligned}
		\int_{\mathbb{\R}^N} \frac{u}{1+|x|^\gamma} d x\leq C+C \sum_{i=1}^{\infty}\left(2^i\right)^{N-\frac{4}{p-1}-\gamma}<+\infty ,
	\end{aligned}
\end{equation*}
for $\gamma>N-\frac{4}{p-1}$.
The proof is complete.
\end{proof}
	Under the assumptions of Proposition \ref{distribu}, for any $x\in\R^N\setminus Z$, we define
	\begin{equation}\label{vx}
		v(x):=c_{N,2}\int_{\R^N}\frac{u^{p}(y)}{|x-y|^{N-4}}dy,
	\end{equation}
$c_{N,2}=\Gamma(\frac{N-4}{2})\left[ 2^{4}\pi^{\frac{N}{2}}\Gamma(2)\right]^{-1} $, where $\Gamma$ is the Riemann Gamma function. Clearly, by \eqref{up4}, $v$ is well-defined and $v\in L^{1}_{loc}(\R^N)$. Indeed, for $R>0$,
\begin{equation*}
	v(x) =c_{N,2}\int_{B_{2R}}\frac{u^{p}(y)}{|x-y|^{N-4}}dy+c_{N,2}\int_{B_{2R}^{c}}\frac{u^{p}(y)}{|x-y|^{N-4}}dy:=v_{1,R}+v_{2,R},
\end{equation*}
it follows from $u^{p}\in L^{1}(B_{2R})$ that $v_{1,R}\in L^{1}(B_{R})$. And we get from \eqref{up4} again that $v_{2,R}\in L^{\infty}(B_{R})$. Thus, we have $v\in L^{1}_{loc}(\R^N)$. Furthermore, combining with \eqref{up4}, we can obtain the following property of $v$.
\begin{lem}\label{vx0}
Under the assumptions of Proposition \ref{distribu}, then we have $v\in L_{0}(\R^N)$, where
$$L_{0}(\R^N):=\left\lbrace v\in L_{loc}^{1}(\R^N)\  \left| \ \int_{R^N}\frac{v}{1+|x|^{N}}dx<\infty\right\rbrace \right. .$$
\end{lem}
\begin{proof}
	We skip the proof for convenience, as it follows a similar approach to Lemma 2.7 in \cite{DY}.
\end{proof}
\begin{thm}\label{integral}
	Under the assumptions of Proposition \ref{distribu}, then $u$ is a solution of the integral equation
		\begin{equation}\label{uvi}
		u(x)=c_{N,2}\int_{\R^N}\frac{u^{p}(y)}{|x-y|^{N-4}}dy, \quad \text{for } x\in\R^N\setminus Z.
	\end{equation}
\end{thm}
\begin{proof}
	 As $v$ is defined in \eqref{vx}, we see that $v\in L^{1}_{loc}(\R^N)$ and it satisfies
	 \begin{equation*}
	 	\Delta^{2} v=u^{p}
	 	\quad
	 	\quad \text{in} \ \mathbb{R}^{N}
	 \end{equation*}
 in the distributional sense, that is,  	\begin{equation}\label{vx2}
 	\int_{\R^N} v\Delta^{2}\varphi dx=	\int_{\R^N}u^{p} \varphi dx, \quad \forall\varphi\in C^{\infty}_{c}(\mathbb{R}^{N}).
 \end{equation}	
In fact, we know $\mathcal{L}^{N}(Z)=0$. For any $\varphi\in C_{c}^{\infty}(\R^N)$, there holds
\begin{equation*}
	\begin{aligned}
		\int_{\mathbb{R}^N} v(x)\Delta^2 \varphi(x) d x & =\int_{\mathbb{R}^N}\left(c_{N, 2} \int_{\mathbb{R}^N} \frac{u^{p}(y)}{|x-y|^{N-4}} d y\right)\Delta^2 \varphi(x) d x \\
		& =\int_{\mathbb{R}^N} u^{p}(y)\left(c_{N, 2} \int_{\mathbb{R}^N} \frac{\Delta^2 \varphi(x)}{|x-y|^{N-4}} d x\right) d y \\
		& =\int_{\mathbb{R}^N} u^{p}(y) \varphi(y) d y,
	\end{aligned}
\end{equation*}
own on the Fubini's theorem. Let $w=u-v$, then it from  Proposition \ref{distribu} and \eqref{vx2} that \begin{equation*}
	\Delta^{2} w=0
	\quad
	\quad \text{in} \ \mathbb{R}^{N}
\end{equation*}
in the distributional sense. Morover, by \eqref{up3} and Lemma \ref{vx0}, we easily see that both $u$ and $v$ belong to $L_{0}(\R^N)$. Hence, we have $w\in L_{0}(\R^N)$, by the Liouville type theorem (see \cite{AGHW}, Lemma 5.7), it follows $w\equiv0$ in $\R^N$. Therefore, we have \eqref{uvi}.
\end{proof}
\textbf{Proof of Theorem \ref{symsub}}. Notice that  $u$ has non-removable singularity in the singular set $ Z\subset\left\lbrace x_{1}=0\right\rbrace $, without loss of generality, we suppose
\begin{equation}\label{limsup}
	\limsup_{\substack{x\in\R^N\setminus Z \\ x_{1}\rightarrow 0}} u(x)=\infty.
\end{equation}
 It follows from Theorem \ref{integral} that $u$ satisfies \eqref{uvi}. Up to a multiplicative constant is omitted for
brevity, next we will apply the method of moving spheres to the integral equation
	\begin{equation*}
	u(y)=\int_{\R^N}\frac{u^{p}(z)}{|y-z|^{N-4}}dz, \quad \text{for } y\in\R^N\setminus Z.
\end{equation*}

\textit{Step 1.} For every $x=(x_{1},0)\in \R^N$ with $x_{1}\neq0$, there exists a real number $\lambda_{2}\in (0, |x|)$ such that for any $0<\lambda<\lambda_{2}$, we have
\begin{equation}\label{uy0}
	u_{x,\lambda}(y)\leq u(y), \quad \forall\ |y-x|\geq\lambda,\quad y\in\R^N\setminus Z,
\end{equation}
where $u_{x,\lambda}(y):=\left(\frac{\lambda}{|y-x|} \right)^{N-4}u(y^{x,\lambda}) $ and $y^{x,\lambda}:=x+\frac{\lambda^{2}(y-x)}{|y-x|^{2}}$.

Firstly, we will show that there exists $0<\lambda_{1}<|x|$ such that for all $0<\lambda<\lambda_{1}$,
\begin{equation}\label{uyy1}
	u_{x,\lambda}(y)\leq u(y), \quad \forall\ \lambda\leq|y-x|<\lambda_{1},\quad y\in\R^N\setminus Z.
\end{equation}
Indeed, we can assume that $|\nabla \ln u|\leq C_{0}$ in $B_{\frac{|x|}{2}}(x)$ for some constant $C_{0}>0$. Then for any $0<r<\lambda_{1}:=\min\left\lbrace \frac{N-4}{2C_{0}},\frac{|x|}{2} \right\rbrace $ and $\theta\in \mathbb{S}^{N-1}$, we have
\begin{equation}\label{ur}
	\begin{aligned}
		\frac{d}{d r}\left(r^{\frac{N-4}{2}} u(x+r \theta)\right) & =r^{\frac{N-4}{2}-1} u(x+r \theta)\left(\frac{N-4}{2}-r \frac{\nabla u \cdot \theta}{u}\right) \\
		& \geq r^{\frac{N-4}{2}-1} u(x+r \theta)\left(\frac{N-4}{2}-C_0 r\right)>0.
	\end{aligned}
\end{equation}
For $0<\lambda\leq|y-x|<\lambda_{1}$, let $\theta=\frac{y-x}{|y-x|}$, $r_{1}=|y-x|$ and $r_{2}=\frac{\lambda^{2}r_{1}}{|y-x|^{2}}$, then by \eqref{ur} we obtain
\begin{equation*}
	r_2^{\frac{N-4}{2}} u(x+r_{2} \theta)<r_{1}^{\frac{N-4}{2}} u(x+r_{1} \theta),
\end{equation*}
which implies
\begin{equation*}
	u_{x,\lambda}(y)\leq u(y), \quad \forall\ 0<\lambda\leq|y-x|<\lambda_{1},\quad y\in\R^N\setminus Z.
\end{equation*}

\medskip

Secondly, we show that there exists $0<\lambda_{2}<\lambda_{1}<|x|$ such that for all $0<\lambda<\lambda_{2}$,
\begin{equation}\label{uyy2}
	u_{x,\lambda}(y)\leq u(y), \quad \forall \ |y-x|\geq\lambda_{1},\quad y\in\R^N\setminus Z.
\end{equation}
In fact,
using the Fatou Lemma, we obtain
\begin{equation*}
	\liminf _{\substack{y \in \mathbb{R}^N \backslash Z \\|y| \rightarrow \infty}}|y|^{N-4} u(y)=\liminf _{\substack{y \in \mathbb{R}^N \backslash Z \\|y| \rightarrow \infty}} \int_{\mathbb{R}^N} \frac{|y|^{N-4} u^{p}(z)}{|y-z|^{N-4}} d z \geq \int_{\mathbb{R}^N} u^{p}(z) d z>0,
\end{equation*}
from which there exist $c_1>0$ and $R_{1}>0$ such that
\begin{equation}\label{uy1}
	u(y)\geq\frac{c_{1}}{|y|^{N-4}} \quad \text{for} \ |y|\geq R_{1},\ y\in\R^N\setminus Z.
\end{equation}
On the other hand, for $y\in\R^N\setminus Z$ and $ |y|< R_{1},$ we get
\begin{equation}\label{uy2}
	u(y)\geq \int_{B_{R_{1}}}\frac{u^{p}(z)}{|y-z|^{N-4}}dz\geq (2R_{1})^{4-N}\int_{B_{R_{1}}}u^{p}(z)dz>0.
\end{equation}
Thus, by \eqref{uy1} and \eqref{uy2}, we know that there exists $C > 0$ such that
\begin{equation}\label{uy3}
	u(y)\geq\frac{C}{|y-x|^{N-4}} \quad \text{for} \ |y-x|\geq \lambda_{1},\ y\in\R^N\setminus Z.
\end{equation}
Therefore, by \eqref{uy3},  we choose $\lambda_{2}\in(0,\lambda_{1})$ sufficiently small, for $0<\lambda<\lambda_{2}$, there holds
\begin{equation*}
	u_{x,\lambda}(y)=\left(\frac{\lambda}{|y-x|} \right)^{N-4}u(y^{x,\lambda})\leq\left(\frac{\lambda}{|y-x|} \right)^{N-4}\sup\limits_{B_{\lambda_{1}}(x)}u\leq u(y), \quad \forall\ |y-x|\geq\lambda_{1},\quad y\in\R^N\setminus Z.
\end{equation*}
It is easy to see that \eqref{uy0} follows by \eqref{uyy1} and \eqref{uyy2}.

\medskip

Define
\begin{equation}
		\bar{\lambda}(x)  :=\sup \left\{0<\mu \leq|x|\ \left|\ u_{x, \lambda}(y) \leq u(y), \ \forall \ \right| y-x \mid \geq \lambda,\ y \in \mathbb{R}^N \setminus Z,\right.
		 \forall\ 0<\lambda<\mu\},
\end{equation}
then  $\bar{\lambda}(x)>0$ is well defined and $\bar{\lambda}\leq |x|$.

\medskip

 \textit{Step 2.} We prove $	\bar{\lambda}(x)=|x|$, for any $x=(x_{1},0)\in \R^N$ with $x_{1}\neq0$.

Argue by contradiction and suppose that $\bar{\lambda}<|x|$.  For simplicity, we denote $\bar{\lambda}=\bar{\lambda}(x)$.
Claim that there exists $\varepsilon>0$ such that  for any $\lambda\in\left[\bar{\lambda},\bar{\lambda}+\varepsilon \right) $,
\begin{equation*}
u_{x, \lambda}(y) \leq u(y), \ \forall \ | y-x| \geq \lambda,\ y \in \mathbb{R}^N \setminus Z.
\end{equation*}
Indeed, by the definition of $\bar{\lambda}$, there holds
\begin{equation}\label{ux1}
	u_{x, \bar{\lambda}}(y) \leq u(y), \ \forall \ | y-x | \geq \bar{\lambda},\ y \in \mathbb{R}^N \setminus Z.
\end{equation}
 Then we have $	u_{x, \bar{\lambda}}(y) \not\equiv u(y)$, for $ | y-x | \geq \bar{\lambda}$ with $ y \in \mathbb{R}^N \setminus Z.$ Otherwise, $	u_{x, \bar{\lambda}}(y)\equiv u(y)$, it is impossible by \eqref{limsup}.  For $\bar{\lambda}\leq\lambda<|x|$ and $ y \in \mathbb{R}^N \setminus Z$, by direct calculation, it follows
 \begin{equation*}
u_{x, \lambda}(y)=\int_{ \R^N} \frac{u_{x,\lambda}^{p}(z)}{|y-z|^{N-4}}\left(\frac{\lambda}{|z-x|}\right)^{\tau} d z,
 \end{equation*}
from which we have
\begin{equation*}
		u(y)-u_{x, \lambda}(y)=\int_{|z-x| \geq \lambda} K(x, \lambda ; y, z)\left[u^p(z)-\left(\frac{\lambda}{|z-x|}\right)^{\tau} u_{x, \lambda}^p(z)\right] d z,
\end{equation*}
where
\begin{equation*}
	K(x, \lambda ; y, z)=\frac{1}{|y-z|^{N-4}}-\left(\frac{\lambda}{|y-x|}\right)^{N-4} \frac{1}{\left|y^{x, \lambda}-z\right|^{N-4}} \quad\text{and} \quad \tau=N+4-p(N-4)\geq0 .
\end{equation*}
Moreover,
\begin{equation*}
		K(x, \lambda ; y, z)>0,\quad \forall\ |y-x|,|z-x|>\lambda>0.
\end{equation*}
Hence, by the positivity of $K(x, \lambda ; y, z)$ and $\tau\geq0$, it follows
\begin{equation}\label{uintegral}
	\begin{aligned}
		u(y)-u_{x, \lambda}(y)&	=\int_{|z-x| \geq \lambda} K(x, \lambda ; y, z)\left[u^p(z)- u_{x, \lambda}^p(z)\right] d z+\int_{|z-x| \geq \lambda} K(x, \lambda ; y, z)\left[1-\left(\frac{\lambda}{|z-x|}\right)^{\tau}\right] u_{x, \lambda}^p(z) d z\\
		&\geq\int_{|z-x| \geq \lambda} K(x, \lambda ; y, z)\left[u^p(z)- u_{x, \lambda}^p(z)\right] d z.
	\end{aligned}
\end{equation}
Since $\mathcal{L}^{N}(Z)=0$, we obtain by \eqref{ux1} and \eqref{uintegral} that
\begin{equation*}
	u_{x, \bar{\lambda}}(y) < u(y), \quad \forall \ | y-x | \geq \bar{\lambda},\ y \in \mathbb{R}^N \setminus Z.
\end{equation*}
Applying the Fatou Lemma again, together with $\eqref{uintegral}$, we get
\begin{equation*}
	\begin{aligned}
	&\liminf _{\substack{y \in \mathbb{R}^N \backslash Z \\|y| \rightarrow \infty}}|y-x|^{N-4} \left( u(y)-u_{x, \bar{\lambda}}(y)\right) \\
	\geq~&\liminf _{\substack{y \in \mathbb{R}^N \backslash Z \\|y| \rightarrow \infty}}\int_{|z-x| \geq \bar{\lambda}}|y-x|^{N-4} K(x, \bar{\lambda }; y, z)\left[u^p(z)- u_{x, \bar{\lambda}}^p(z)\right] d z\\
	\geq~&\int_{|z-x| \geq \bar{\lambda}}\left[ 1-\left(\frac{\bar{\lambda}}{|z-x|} \right)^{N-4} \right] \left[u^p(z)- u_{x, \bar{\lambda}}^p(z)\right] d z>0.
		\end{aligned}
\end{equation*}
Thus, there exist $c_{2}>0$ and $R_{2}>0$ such that
\begin{equation}\label{ui1}
u(y)-u_{x, \bar{\lambda}}(y)\geq\frac{c_{2}}{|y-x|^{N-4}}, \quad \forall \ | y-x | \geq R_{2},\ y \in \mathbb{R}^N \setminus Z.
\end{equation}
By the definition of $K(x, \lambda ; y, z)$, for $| y-x | \geq\bar{\lambda}+1$ and $|z-x|\geq\bar{\lambda}+2$,  there holds
\begin{equation}\label{kx}
	\frac{\delta_{1}}{|y-z|^{N-4}}\leq K(x, \lambda ; y, z)\leq \frac{1}{|y-z|^{N-4}},
\end{equation}
for some $\delta_{1}\in(0,1)$.
 Consequently, for $\bar{\lambda}+1\leq| y-x | < R_{2},\ y \in \mathbb{R}^N \setminus Z$ and $|z-x|\geq\bar{\lambda}+2$,  we deduce by \eqref{uintegral} and \eqref{kx} that
 \begin{equation*}
 	\begin{aligned}
 u(y)-u_{x, \bar{\lambda}}(y)
 &\geq\int_{\bar{\lambda}+2\leq|z-x| \leq \bar{\lambda}+4} K(x, \bar{\lambda} ; y, z)\left[u^p(z)- u_{x,\bar{\lambda}}^p(z)\right] d z\\
 &\geq\int_{\bar{\lambda}+2\leq|z-x| \leq \bar{\lambda}+4} \frac{\delta_{1}}{|y-z|^{N-4}}\left[u^p(z)- u_{x,\bar{\lambda}}^p(z)\right] d z\\
 &\geq C_{2}\int_{\bar{\lambda}+2\leq|z-x| \leq \bar{\lambda}+4}\left[u^p(z)- u_{x,\bar{\lambda}}^p(z)\right] d z>0
 	\end{aligned}
 \end{equation*}
with some $C_{2}>0$. This, together with \eqref{ui1}, we know that there exists $\varepsilon_{1}>0$ small enough such that
\begin{equation}\label{ui2}
	u(y)-u_{x, \bar{\lambda}}(y)\geq\frac{\varepsilon_1}{|y-x|^{N-4}}, \quad \forall \ | y-x | \geq \bar{\lambda}+1,\ y \in \mathbb{R}^N \setminus Z.
\end{equation}
Therefore, combining with \eqref{ui2}, there exists $0<\varepsilon_{2}<\varepsilon_{1}$ such that for $\bar{\lambda}\leq\lambda\leq\bar{\lambda}+\varepsilon_{2}<|x|$,
\begin{equation}\label{ui3}
	\begin{aligned}
	u(y)-u_{x, \lambda}(y) & \geq \frac{\varepsilon_1}{|y-x|^{N-4}}+u_{x, \bar{\lambda}}(y)-u_{x, \lambda}(y) \\
		& \geq \frac{\varepsilon_1}{2|y-x|^{N-4}}, \quad \forall\ |y-x| \geq \bar{\lambda}+1, \quad y \in \mathbb{R}^N \setminus Z.
	\end{aligned}
\end{equation}
On the other hand, for $\varepsilon\in(0,\varepsilon_2)$ to be determined later, by \eqref{ux1}, \eqref{uintegral} and \eqref{ui3}, for any $\lambda\in\left[\bar{\lambda},\bar{\lambda}+\varepsilon \right) $ and $ \lambda\leq |y-x| < \bar{\lambda}+1 $, $ y \in \mathbb{R}^N \setminus Z$, we have
\begin{equation}\label{ui4}
	\begin{aligned}
		u(y)-u_{x, \lambda}(y)
		\geq & \int_{\lambda \leq|z-x| \leq \bar{\lambda}+1} K(x, \lambda ; y, z)\left[u^{p}(z)-u_{x, \lambda}^{p}(z)\right] d z \\
		& +\int_{\bar{\lambda}+2 \leq|z-x| \leq \bar{\lambda}+3} K(x, \lambda ; y, z)\left[u^{p}(z)-u_{x, \lambda}^{p}(z)\right]  d z \\
		\geq & \int_{\lambda \leq|z-x| \leq \bar{\lambda}+\varepsilon} K(x, \lambda ; y, z)\left[u^{p}(z)-u_{x, \lambda}^{p}(z)\right] d z \\
		& +\int_{\bar{\lambda}+\varepsilon<|z-x| \leq \bar{\lambda}+1} K(x, \lambda ; y, z)\left[u_{x, \bar{\lambda}}^{p}(z)-u_{x, \lambda}^{p}(z)\right] d z \\
		& +\int_{\bar{\lambda}+2 \leq|z-x| \leq \bar{\lambda}+3} K(x, \lambda ; y, z)\left[u^{p}(z)-u_{x, \lambda}^{p}(z)\right] d z.
	\end{aligned}
\end{equation}
It is worth mentioning that there exists $C>0$ independent of $\varepsilon$ such that for any $\lambda\in\left[\bar{\lambda},\bar{\lambda}+\varepsilon \right) $,
\begin{equation*}
	\left|u^{p}(z)-u_{x, \lambda}^{p}(z)\right| \leq C\left(|z-x|-\lambda \right) , \quad \forall \  \lambda \leq|z-x| \leq \bar{\lambda}+\varepsilon ,\ z \in \mathbb{R}^N \setminus Z,
\end{equation*}
and
\begin{equation*}
	\left|u_{x, \bar{\lambda}}^{p}(z)-u_{x, \lambda}^{p}(z)\right| \leq C(\lambda-\bar{\lambda}) \leq C \varepsilon, \quad \forall \ \lambda \leq|z-x| \leq \bar{\lambda}+1 ,\ z \in \mathbb{R}^N \setminus Z,
\end{equation*}
since $\|u\|_{C^{1}(B_{|x|}(x))}\leq C$. Moreover, it follows by \eqref{ui3} that there exists $\delta_{2}>0$ such that for $\lambda\in\left[\bar{\lambda},\bar{\lambda}+\varepsilon \right) $,
\begin{equation*}
u^{p}(z)-u_{x, \lambda}^{p}(z) \geq \delta_{2}, \quad \forall \ \bar{\lambda} +2\leq \lambda \leq|z-x| \leq \bar{\lambda}+3 ,\ z \in \mathbb{R}^N \setminus Z.
\end{equation*}
Hence, using above three estimates, by \eqref{ui4}, we have
\begin{equation}\label{ui5}
	\begin{aligned}
		u(y)-u_{x, \lambda}(y)
		\geq& -C\int_{\lambda \leq|z-x| \leq \bar{\lambda}+\varepsilon} K(x, \lambda ; y, z) \left(|z-x|-\lambda \right) d z
		-C\varepsilon\int_{\bar{\lambda}+\varepsilon<|z-x| \leq \bar{\lambda}+1} K(x, \lambda ; y, z) d z\\
		&+\delta_{2}\int_{\bar{\lambda}+2 \leq|z-x| \leq \bar{\lambda}+3} K(x, \lambda ; y, z) d z \\
		=&  -C\int_{\lambda \leq|z| \leq \bar{\lambda}+\varepsilon} K(0, \lambda ; y-x, z)\left(|z|-\lambda \right) d z
		-C\varepsilon\int_{\bar{\lambda}+\varepsilon<|z| \leq \bar{\lambda}+1} K(0, \lambda ; y-x, z) d z\\
		&+\delta_{2}\int_{\bar{\lambda}+2 \leq|z| \leq \bar{\lambda}+3} K(0, \lambda ; y-x, z) d z .
	\end{aligned}
\end{equation}
 Notice that $ K(0, \lambda ; y-x, z)=0$ for $|y-x|=\lambda$ and
\begin{equation*}
	(y-x)\cdot\nabla_{y}K(0, \lambda ; y-x, z)=(N-4)|y-x-z|^{2-N}(|z|^{2}-|y-x|^{2})>0
\end{equation*}
for $|y-x|=\lambda$, $|z|\geq\bar{\lambda}+2$.
By virtue of the positive and smoothness of $ K(0, \lambda ; y-x, z)$, for $ \bar{\lambda}\leq \lambda\leq|y-x| < \bar{\lambda}+1 $ and $\bar{\lambda}+2<|z| \leq M<\infty$, there holds
\begin{equation}\label{kx1}
	\frac{\delta_{3}(|y-x|-\lambda)}{|y-x-z|^{N-4}}\leq K(0, \lambda ; y-x, z)\leq 	\frac{\delta_{4}(|y-x|-\lambda)}{|y-x-z|^{N-4}},
\end{equation}
where $M$ and $0<\delta_{3}\leq\delta_{4}<\infty$ are positive constants. In addition, if $M$ is large enough, for $ \lambda\leq|y-x| < \bar{\lambda}+1 $ and $|z| \geq M$, we obtain
\begin{equation*}
	0\leq c_{3}\leq (y-x)\cdot\nabla_{y}\left( |y-x-z|^{N-4}K(0, \lambda ; y-x, z)\right) \leq c_{4}<\infty.
\end{equation*}
Thus, \eqref{kx1} also holds for $ \lambda\leq|y-x| < \bar{\lambda}+1 $ and $|z| \geq M$.
Using \eqref{kx1}, for $ \lambda\leq |y-x| < \bar{\lambda}+1 $, we deduce that
\begin{equation*}
	\begin{aligned}
	\int_{\bar{\lambda}+2 \leq|z| \leq \bar{\lambda}+3} K(0, \lambda ; y-x, z) d z \geq\delta_{3}(|y-x|-\lambda)\int_{\bar{\lambda}+2 \leq|z| \leq \bar{\lambda}+3} 	\frac{1}{|y-x-z|^{N-4}} d z\geq C(|y-x|-\lambda).
	\end{aligned}
\end{equation*}
Following the estimates of the integral of $K$ in the proof of Proposition 3.2 in \cite{JX}, for $ \lambda\leq |y-x| < \bar{\lambda}+1 $, there holds
\begin{equation*}
	\begin{aligned}
	 \int_{\lambda \leq|z| \leq \bar{\lambda}+\varepsilon} K(0, \lambda ; y-x, z)\left(|z|-\lambda \right) d z
	 	 \leq&~\left|\int_{\lambda \leq|z| \leq \lambda+\varepsilon}\left(\frac{|z|-\lambda}{|y-x-z|^{N-4}}-\frac{|z|-\lambda}{\left|(y-x)^{0, \lambda}-z\right|^{N-4}}\right) d z\right| \\
	 	& ~+\varepsilon \int_{\lambda \leq|z| \leq \lambda+\varepsilon}\left|\left(\frac{\lambda}{|y-x|}\right)^{N-4}-1\right| \frac{1}{\left|(y-x)^{0, \lambda}-z\right|^{N-4}} d z \\
	 	 \leq&~ C(|y-x|-\lambda) \varepsilon^{\frac{4}{N}}+C \varepsilon(|y-x|-\lambda) \\
	 	 \leq& ~C(|y-x|-\lambda) \varepsilon^{\frac{4}{N}},
	 \end{aligned}
\end{equation*}
and
\begin{equation*}
		\begin{aligned}
\int_{\bar{\lambda}+\varepsilon<|z| \leq \bar{\lambda}+1} K(0, \lambda ; y-x, z) d z
	 \leq&~\left|\int_{\bar{\lambda}+\varepsilon<|z| \leq \bar{\lambda}+1}\left(\frac{1}{|y-x-z|^{N-4}}-\frac{1}{\left|(y-x)^{0, \lambda}-z\right|^{N-4}}\right) d z\right| \\
&~ +\int_{\bar{\lambda}+\varepsilon<|z| \leq \bar{\lambda}+1}\left|\left(\frac{\lambda}{|y-x|}\right)^{N-4}-1\right| \frac{1}{\left|(y-x)^{0, \lambda}-z\right|^{N-4}} d z \\
\leq &~C(\varepsilon^{3}+|\ln\varepsilon|+1)(|y-x|-\lambda) .
\end{aligned}
\end{equation*}
Consequently, inserting these three integral of $K$ into \eqref{ui5},  for $ \lambda\leq |y-x| < \bar{\lambda}+1 $, we yield
\begin{equation}\label{ui6}
		u(y)-u_{x, \lambda}(y)\geq C(|y-x|-\lambda)\left( \delta_{2}-\varepsilon^{\frac{4}{N}}-\varepsilon(\varepsilon^{3}+|\ln\varepsilon|+1)\right) \geq 0,
\end{equation}
if $\varepsilon$ is sufficiently small. Combining \eqref{ui3} and \eqref{ui6}, our claim is proved, which contradicts to the definition of $\bar{\lambda}$ provided $\bar{\lambda}<|x|$. Thus, we have $\bar{\lambda}=|x|$, which implies for every $x=(x_{1},0)\in \R^N$ with $x_{1}\neq0$,
\begin{equation}\label{uxx}
	u_{x, \lambda}(y) \leq u(y), \ \forall \ | y-x| \geq \lambda,\ y \in \mathbb{R}^N \setminus Z,\quad \forall \ 0<\lambda<|x|.
\end{equation}
\textit{Step 3.} We will show that $u$ is symmetric with respect to the  hyperplane $ \left\lbrace x_{1}=0\right\rbrace $.
Without loss of generality, let  $x=(x_{1},0)\in \R^N$ with $x_{1}>0$. For any $0<a<x_1$, take $\lambda=x_{1}-a$. For any $y=(y_{1},w)\in\R^N\setminus Z$ with $y_{1}\leq a$ and $w\in\R^{N-1}$, then we have $|y-x|\geq\lambda$. Since $\frac{\lambda^{2}}{|y-x|^2}=\frac{(x_{1}-a)^{2}}{w^2+(x_{1}-y_{1})^{2}}\rightarrow1$ as $x_1\rightarrow \infty$, we obtain that
\begin{equation*}
\begin{aligned}
	y^{x,\lambda}=x+\frac{\lambda^{2}(y-x)}{|y-x|^{2}} &=\left(\frac{\lambda^{2}y_1}{|y-x|^{2}}+ \frac{\left(|y-x|^2-\lambda^2\right) x_{1}}{|y-x|^2},\frac{\lambda^{2}w}{|y-x|^{2}}\right)\\
&=\left(\frac{\lambda^{2}y_1}{|y-x|^{2}}+ \frac{\left(|y|^2+2x_{1}(a-y_{1})-a^2\right) x_{1}}{|y-x|^2},\frac{\lambda^{2}w}{|y-x|^{2}}\right)\\
&\rightarrow \left(2a-y_{1},w\right)
\end{aligned}
\end{equation*}
as $x_1\rightarrow \infty$.
Therefore, letting $x_{1}\rightarrow\infty$, it follows by \eqref{uxx} that for any $a>0$, there holds
\begin{equation*}
	u(y_{1},w)\geq u(2a-y_{1},w), \forall \ y_{1}\leq a.
\end{equation*}
Similarly, we have
\begin{equation*}
	u(y_{1},w)\geq u(2a-y_{1},w), \forall \ y_{1}\geq a, \ \forall\ a<0,
\end{equation*}
from which we conclude that $u$ is symmetric with respect to the  hyperplane $ \left\lbrace x_{1}=0\right\rbrace $. Moreover, $u$ is decreasing in the $x_{1}$-direction in $\left\lbrace x_{1}>0 \right\rbrace$. The proof is complete.

\end{document}